\documentclass[fleqn,preprint,3p,a4paper]{elsarticle}

\usepackage{amssymb}
\usepackage{amsmath}
\usepackage{amsthm}
\usepackage{dcolumn}
\usepackage{endnotes}
\usepackage{tabularx}
\usepackage[matrix,arrow]{xy}
\usepackage{wasysym}

\newtheorem{theorem}{Theorem}[section]
\newtheorem{proposition}[theorem]{Proposition}
\newtheorem{lemma}[theorem]{Lemma}
\newtheorem{corollary}[theorem]{Corollary}

\theoremstyle{definition}
\newtheorem{definition}[theorem]{Definition}
\newtheorem{example}[theorem]{Example}
\newtheorem{remark}[theorem]{Remark}
\newtheorem{question}[theorem]{Question}

\newcommand{\ir}{{\mathsf{Irr}}}

\newcommand{\mn}{\mathbb N}

\newcommand{\cl}{{\rm cl}}
\newcommand{\ii}{{\rm int}}

\newcommand{\ua}{\mathord{\uparrow}}
\newcommand{\da}{\mathord{\downarrow}}

\newcommand{\mk}{\mathord{\mathsf{K}}}

\journal{}

\begin{document}

\begin{frontmatter}



\title{On sobriety of Scott topology on dcpos\tnoteref{t1}}
\tnotetext[t1]{This research was supported by the National Natural Science Foundation of China (Nos. 12071199, 11661057).}

\author[X. Xu]{Xiaoquan Xu}
\ead{xiqxu2002@163.com}
\address[X. Xu]{Fujian Key Laboratory of Granular Computing and Applications,

School of Mathematics and Statistics, Minnan Normal University, Zhangzhou 363000, China}

\begin{abstract}
In this paper, we mainly investigate the conditions under which the Scott topology on the product of two posets is equal to the product of the individual Scott topologies and under which the Scott topology on a dcpo is sober. Some such conditions are given.

\end{abstract}





\end{frontmatter}


\section{Introduction}

Sobriety is probably the most important and useful property of $T_0$ spaces (see [1-2, 7-14, 17-21, 23-29, 39-51]). It has been used in the characterizations of spectral spaces of commutative rings by Hochster in \cite{Hochster-1969} and the $T_0$ spaces which are determined by their open set lattices (see \cite{Drake-Thron-1965, Thron-1962}). In domain theory and non-Hausdorff topology, the Scott topology on posets (especially, dcpos) is the most important topology (see \cite{Abramsky-Jung-1994, GHKLMS-2003, Goubault-2013}). The following problem is one of the oldest and best-known open problems in domain theory (see \cite[page 155]{GHKLMS-2003}).

\vskip 0.1cm

$\mathbf{Problem~1.}$ Characterise those dcpos for which the Scott topology
is sober.

\vskip 0.1cm

This problem may be very difficult and has been open for decades. A relatively simpler problem is the following one.

\vskip 0.1cm

$\mathbf{Problem~2.}$ Under what conditions does the Scott
topology on a dcpo $P$ is sober?

\vskip 0.1cm

It is well-known that the Scott space of a continuous
domain is sober (see \cite{GHKLMS-2003}). Furthermore, it was proved in \cite{Gierz-Lawson-Stralka-1983} that the Scott space of a quasicontinuous
domain is sober. For a complete lattice $L$, the sobriety of Scott space $\Sigma~\!\!L$ is closely related with the property of $\Sigma~\!\!(L\times L)=\Sigma~\!\!L\times \Sigma~\!\!L$ (cf. \cite[Section II-4]{GHKLMS-2003}). The Scott
topology on $L\times L$ is strictly finer than the
product of the individual Scott topologies in general. The following question arises naturally.

\vskip 0.1cm

$\mathbf{Problem~3.}$ For a dcpo $P$, under what circumstances does the Scott
topology on $P\times P$ agree with the
product of the individual Scott topologies?

\vskip 0.1cm

Problem 2 and Problem 3 have attracted the attention of many researchers \cite{Erne-1985, GHKLMS-2003, Hertling-2022, Isbell-1982, Isbell-1985, Miao-Xi-Li-Zhao-2022, Xu-2016-2} and several important results have been obtained (see Sections 3-5 below).

In this paper, inspired by the recent work of Hertling \cite{Hertling-2022} and Miao et al. \cite{Miao-Xi-Li-Zhao-2022}, we continue to study Problem 2 and Problem 3. The paper is organized as follows:

In Section 2, we present some fundamental concepts and results about lattice-ordered structures and intrinsic topologies on posets that will be used in the whole paper.

Section 3 is dedicated to reviewing some related properties of sober spaces, well-filtered spaces and Scott topologies (especially, the Scott topology on the product of two posets).

In Section 4, we mainly investigate the conditions under which the Scott topology on the product of two posets is equal to the product of the individual Scott topologies. The concepts of $c$-posets and $\ell f_{\omega}$-posets are introduced. Based on Lawson's work on $\ell c_{\omega}$-spaces and $\ell c_{\omega}$-dcpos, some properties of $c$-posets and $\ell f_{\omega}$-posets are studied, and the implications of (1) $\Rightarrow$ (2) $\Rightarrow$ (3) $\Rightarrow$ (4) $\Rightarrow$ (5) are proved for the following conditions:  (1) the set $\mathrm{Id} (P)$ of all ideals of $P$ is countable; (2) the Scott space $\Sigma~\!\!P$ is a $c$-space; (3) $P$ is an $\ell f_{\omega}$-poset; (4) $P$ is an $\ell c_{\omega}$-poset; (5) $\Sigma~\!\!(P\times P)=\Sigma~\!\!P\times\Sigma~\!\!P$. Three related examples are presented.

In Section 5, using the property R, first introduced in \cite{Xu-2016-2}, we investigate the conditions under which the Scott
topology on a dcpo is sober. It is shown that for a poset $P$, the upper semi-compactness of the Lawson space $(P,\lambda(P))$ implies the property R of the Scott space $\Sigma~\!\!P$, and if a $T_0$ space $X$ is well-filtered and coherent, then $X$ is has property R. For the Scott space $\Sigma~\!\!P$ of a poset $P$, some characterizations of the property R are obtained. Based on them, it is proved that if a poset $P$ has property R and $\Sigma~\!\!(P\times P)=\Sigma~\!\!P\times\Sigma~\!\!P$, then $\Sigma~\!\!P$ is sober. Some applications of this result are given. Finally, we pose two related questions.

\section{Preliminaries}

In this section, we briefly recall some fundamental concepts and basic results that will be used in the paper. For further details, we refer the reader to \cite{Engelking-1989, GHKLMS-2003, Goubault-2013}.

For a poset $P$ and $A\subseteq P$, let
$\mathord{\downarrow}A=\{x\in P: x\leq  a \mbox{ for some }
a\in A\}$ and $\mathord{\uparrow}A=\{x\in P: x\geq  a \mbox{
	for some } a\in A\}$. For  $x\in P$, we write
$\mathord{\downarrow}x$ for $\mathord{\downarrow}\{x\}$ and
$\mathord{\uparrow}x$ for $\mathord{\uparrow}\{x\}$. The set $A$
is called a \emph{lower set} (resp., an \emph{upper set}) if
$A=\mathord{\downarrow}A$ (resp., $A=\mathord{\uparrow}A$). Let $P^{(<\omega)}=\{F\subseteq P : F \mbox{~is a nonempty finite set}\}$ and $\mathbf{Fin} P=\{\uparrow F :$ $F\in P^{(<\omega)}\}$.
 For a nonempty subset $A$ of $P$, define $\mathrm{max}(A)=\{a\in A : a \mbox{~ is a maximal element of~} A\}$ and $\mathrm{min}(A)=\{a\in A : a \mbox{~ is a minimal element of~} A\}$.

For a set $X$ and $A, B\subseteq X$, $A\subset B$ means that $A\subseteq B$ but $A\neq B$, that is, $A$ is a proper subset of $B$. Let $|X|$ be the cardinality of $X$ and $2^X$ the set of all subsets of $X$. The set of all natural numbers is denoted by $\mathbb{N}$. When $\mathbb{N}$ is regarded as a poset (in fact, a chain), the order on $\mathbb{N}$ is the usual order of natural numbers. Let $\omega=|\mathbb{N}|$.

    A poset $P$ is called an (\emph{inf}) \emph{semilattice} if for any two elements
$a, b\in Q$, \emph{inf}$\{a, b\}=a\wedge b$ exists in $Q$. Dually, $Q$ is a \emph{sup semilattice} if for any two elements $a, b\in Q$, \emph{sup}$\{a, b\}=a\vee b$ exists in $Q$. $P$  is called \emph{bounded complete}, if every subset that is bounded above
has a sup (i.e., the least upper bound). In particular, a bounded complete poset has a smallest
element, the least upper bound of the empty set. $P$ is called a \emph{complete lattice} if every subset has a sup and an inf.
A totally ordered complete lattice is called a \emph{complete chain}.

\begin{definition}\label{def-complete-Heyting-algebra} A \emph{frame} (or a \emph{complete Heyting algebra} (cHa)) is a complete lattice which satisfies the following infinite distributive law:
\vskip 3mm
(ID)  \qquad \qquad \qquad \qquad \qquad \qquad  $x\wedge (\bigvee A)=\bigvee_{a\in A}x\wedge a$,
\vskip 3mm
\noindent for all elements $x$ and all subsets $A$. Such a
complete lattice is also called a \emph{frame} (see \cite{Johnstone-1993}). Clearly, every
complete Heyting algebra is distributive when viewed as lattice.

\end{definition}

For the following definition, we refer the reader to \cite{Johnstone-1993}.

\begin{definition}\label{def-spatial-frame}  An element $p$ of a meet-semilattice $S$ is a \emph{prime element} if for any $a, b\in S$, $a\wedge b\le p$ implies $a\le p$ or $b\le p$. A frame $A$ is called \emph{spatial} if every element of $A$ can be expressed as a  meet of prime elements.
\end{definition}

\begin{lemma}\label{spatial-frame=lattice-of-open-sets} (\cite{Johnstone-1993}) A complete lattice $L$ is a spatial frame iff it is isomorphic to the lattice of all open subsets of some topological space.
\end{lemma}

\begin{definition}\label{def-CD-lattice}  A \emph{completely distributive lattice} is a complete lattice which satisfies the following completely distributive law:
\vskip 3mm
(CD)  \qquad \qquad \qquad \qquad \qquad \qquad  $\bigwedge\limits_{i\in I} \bigvee A_i=\bigvee\limits_ {\varphi\in \Pi_{i\in I}A_i} \bigwedge \varphi(I)$,
\vskip 3mm
\noindent for all family $\{A_i : i\in I\}\subseteq 2^X$.
\end{definition}

It is well-known that the completely distributivity is self-dual, that is, a complete lattice $L$ is completely distributive iff its dual $L^{op}$ is completely distributive (see \cite{Raney-1952, Raney-1960}).

 A nonempty subset $D$ of a poset $P$ is \emph{directed} if every two
elements in $D$ have an upper bound in $D$. $P$ is said to be \emph{directed} if $P$ itself is directed. The set of all directed sets of $P$ is denoted by $\mathcal D(P)$. The poset $P$ is called a \emph{directed complete poset}, or \emph{dcpo} for short, if for any
$D\in \mathcal D(P)$, $\vee D$ exists in $P$.

  For a $T_0$ space $X$, we use $\leq_X$ to denote the \emph{specialization order} of $X$: $x\leq_X y$ if{}f $x\in \overline{\{y\}}$. Clearly, all open sets  (resp., closed sets) of $X$ are upper sets (resp., lower sets) of $X$. A subset $A$ of $X$ is called \emph{saturated} if $A$ equals the intersection of all open sets containing it (equivalently, $A$ is an upper set in the specialization order). Let $\mathcal O(X)$ (resp., $\mathcal C(X)$) be the set of all open subsets (resp., closed subsets) of $X$, and let $\mathcal D_c(X)=\{\overline{D} : D\in \mathcal D(X)\}$.

  In what follows, when a $T_0$ space $X$ is considered as a poset, the order always refers to the specialization order if no other explanation. We will use $\Omega X$ or even $X$ to denote the poset $(X, \leq_X)$.

\begin{definition}\label{def-upper-semiclosed} Let $P$ be a poset equipped with a topology. The partial
order is said to be \emph{upper semiclosed} if each $\ua x
$ is closed.
\end{definition}

\begin{definition}\label{def-upper-semicompact}
A topological space $X$ with a partially order is called \emph{upper semicompact}, if $\ua x$ is compact for any $x\in X$, or equivalently, if $\ua x\cap A$ is compact for any $x\in X$ and $A\in \mathcal C(X)$.
\end{definition}

A subset $U$ of a poset $P$ is said to be \emph{Scott open} if (i) $U=\mathord{\uparrow}U$, and (ii) for any directed subset $D$ with
$\vee D$ existing, $\vee D\in U$ implies $D\cap
U\neq\emptyset$. All Scott open subsets of $P$ form a topology,
called the \emph{Scott topology} on $P$ and
denoted by $\sigma(P)$. The space $\Sigma~\!\! P=(P,\sigma(P))$ is called the
\emph{Scott space} of $P$. The \emph{lower topology} on $P$, generated
by $\{P\setminus \ua x : x\in P\}$ (as a subbase), is denoted by $\omega (P)$. Dually, define the \emph{upper topology} on $P$ (generated
by $\{P\setminus \da x : x\in P\}$) and denote it by $\upsilon (P)$.
The topology generated by $\omega (P)\cup\sigma (P)$ is called the \emph{Lawson topology} on $P$ and is denoted by $\lambda (P)$. The upper sets of $P$ form the (\emph{upper}) \emph{Alexandroff topology} $\alpha (P)$.

\begin{lemma}\label{upper-semiclosed-closed} (\cite{GHKLMS-2003})
Let $X$ be a topological space with an upper semiclosed partial order. If $A$
is a compact subset of $X$, then $\da A$ is Scott closed.
\end{lemma}

\begin{lemma}\label{Scott-cont-charac} (\cite{GHKLMS-2003}) Let $P, Q$ be posets and $f : P \longrightarrow Q$. Then the following two conditions are equivalent:
\begin{enumerate}[\rm (1)]
	\item $f$ is Scott continuous, that is, $f : \Sigma~\!\! P \longrightarrow \Sigma~\!\! Q$ is continuous.
	\item For any $D\in \mathcal D(P)$ for which $\vee D$ exists, $f(\vee D)=\vee f(D)$.
\end{enumerate}
\end{lemma}

A $T_0$ space $X$ is called a \emph{d-space} (or \emph{monotone convergence space}) if $X$ (with the specialization order) is a dcpo
 and $\mathcal O(X) \subseteq \sigma(X)$ (cf. \cite{GHKLMS-2003, Wyler-1981}).

For the following definition and related conceptions, please refer to \cite{GHKLMS-2003, Goubault-2013}.

\begin{definition}\label{def-way-below-relation} For a dcpo $P$ and $A, B\subseteq P$, we say $A$ is \emph{way below} $B$, written $A\ll B$, if for each $D\in \mathcal D(P)$, $\vee D\in \ua B$ implies $D\cap \ua A\neq \emptyset$. For $B=\{x\}$, a singleton, $A\ll B$ is
written $A\ll x$ for short. For $x\in P$, let $w(x)=\{F\in P^{(<\omega)} : F\ll x\}$, $\Downarrow x = \{u\in P : u\ll x\}$ and $K(P)=\{k\in P : k\ll k\}$. Points in $K(P)$ are called \emph{compact} elements of $P$.
\end{definition}

\begin{definition}\label{def-continuous-domain} Let $P$ be a dcpo and $X$ a $T_0$ space.
\begin{enumerate}[\rm (1)]
\item $P$ is called a \emph{continuous domain}, if for each $x\in P$, $\Downarrow x$ is directed
and $x=\vee\Downarrow x$.
\item $P$ is called a \emph{quasicontinuous domain}, if for each $x\in P$, $\{\ua F : F\in w(x)\}$ is filtered and $\ua x=\bigcap
\{\ua F : F\in w(x)\}$.
\item $X$ is called \emph{core compact} if $\mathcal O(X)$ is a \emph{continuous lattice}.
\end{enumerate}
\end{definition}

It is well-known that every continuous domain is a quasicontinuous domain but the converse implication does not hold in general (see \cite{GHKLMS-2003}).

For the concepts in the following definition, please refer to \cite{Erne-2018-2, GHKLMS-2003, Heckmann-1992}.

    \begin{definition}\label{def-C-LHC-CC-spaces} Let $X$ be a topological space and $S\subseteq X$.
     \begin{enumerate}[\rm (1)]
     \item $S$ is called \emph{strongly compact} if for any open set $U$ with $S \subseteq U$,
there is a finite set $F$  with $S\subseteq \uparrow  F \subseteq U$).
\item $S$ is called \emph{supercompact} if for
any family $\{U_i : i\in I\}\subseteq \mathcal O(X)$, $S\subseteq \bigcup_{i\in I} U_i$  implies $S\subseteq U_i$ for some $i\in I$.
\item $X$ is called \emph{locally hypercompact} if for each $x\in X$ and each open neighborhood $U$ of $x$, there is a strongly compact set $S$
such that $x\in\ii\,\ua S\subseteq \ua S\subseteq U$ or, equivalently, there is $\ua F\in \mathbf{Fin}X$ such that $x\in\ii\,\ua F\subseteq\ua F\subseteq U$.
\item $X$ is called a $C$-\emph{space} if for each $x\in X$ and each open neighborhood $U$ of $x$, there is a supercompact set $S$
such that $x\in\ii\,\ua S\subseteq \ua S\subseteq U$.
\end{enumerate}
\end{definition}

\begin{remark}\label{core-compact-not-LC} It is easy to verify that if a topological space $X$ is locally compact, then it is core compact (see, e.g., \cite[Examples I-1.7]{GHKLMS-2003}). In \cite{Hofmann-Lawson-1978} (see also \cite[Exercise V-5.25]{GHKLMS-2003}) Hofmann and Lawson gave a second-countable core compact $T_0$ space $X$ in which every compact subset of $X$ has empty interior, and hence it is not locally compact.
\end{remark}

\begin{lemma}\label{supercompact-set-charac} (\cite{Heckmann-Keimel-2013})  For a $T_0$ space and a nonempty saturated subset $K$ of $X$, the following two conditions are equivalent:
\begin{enumerate}[\rm (1)]
\item $K$ is supercompact,
\item $K=\ua x$ for some $x \in X$.
\end{enumerate}
\end{lemma}
\begin{proof} (1) $\Rightarrow$ (2): We show that $K\cap \bigcap_{k\in K}\downarrow k\neq\emptyset$. Assume, on the contrary, that $K\cap \bigcap_{k\in K}\downarrow k=\emptyset$, then $K\subseteq \bigcup_{k\in K} (X\setminus \downarrow k)$. As $K$ is supercompact, there is $k\in K$ such that $K\subseteq X\setminus \downarrow k$, a contradiction, proving that $K\cap \bigcap_{k\in K}\downarrow k\neq\emptyset$. Select an $x\in K\cap \bigcap_{k\in K}\downarrow k$. Then $x$ is the least element of $K$ and hence $K=\ua x$.

(2) $\Rightarrow$ (1): Trivial.

\end{proof}

\begin{lemma}\label{C-space=CD-topology} (\cite{Erne-2009}) For a topological space $X$, the following two conditions are equivalent:
 \begin{enumerate}[\rm (1)]
 \item $X$ is a $C$-space.
 \item For each $U\in \mathcal O(X)$ and $x\in U$, there is $u\in U$ such that $x\in \ii~\ua u\subseteq \ua u\subseteq U$.
 \item  $\mathcal O(X)$ is a completely distributive lattice.
 \end{enumerate}
 \end{lemma}

\begin{lemma}\label{continuous-domain=Scott-topology-CD} (\cite{GHKLMS-2003}) For a dcpo $P$, the following three conditions are equivalent:
\begin{enumerate}[\rm (1)]
\item $P$ is continuous.
\item $\Sigma~\!\! P$ is a C-space (that is, $\sigma (P)$ is a completely distributive lattice).
\item For each $U\in \sigma (P)$ and $x\in U$, there is $u\in U$ such that $x\in \ii~\!_{\sigma (P)}\ua u\subseteq \ua u\subseteq U$.
\end{enumerate}
\end{lemma}

\begin{lemma}\label{quasicontinuou-domain=Scott-topology-LHC} (\cite{Gierz-Lawson-Stralka-1983, Heckmann-1992}) Let $P$ be a dcpo. Then
\begin{enumerate}[\rm (1)]
\item $P$ is quasicontinuous,
\item $\Sigma ~\!\! P$ is locally hypercompact (that is, $\sigma (P)$ is a hypercontinuous lattice),
\item For each $U\in \sigma (P)$ and $x\in U$, there is  $\ua F\in \mathbf{Fin}X$ such that $x\in\ii\,\ua F\subseteq\ua F\subseteq U$.
\end{enumerate}
\end{lemma}

 \section{Sober spaces and well-filtered spaces}\label{sober-space-and-well-filtered-spaces}

For a $T_0$ space $X$ and a nonempty subset $A$ of $X$, $A$ is \emph{irreducible} if for any $\{F_1, F_2\}\subseteq \mathcal C(X)$, $A \subseteq F_1\cup F_2$ implies $A \subseteq F_1$ or $A \subseteq  F_2$.  Denote by $\ir(X)$ (resp., $\ir_c(X)$) the set of all irreducible (resp., irreducible closed) subsets of $X$. Clearly, every subset of $X$ that is directed under $\leq_X$ is irreducible.

\begin{remark}\label{irr-open-charac}  For a $T_0$ space $X$ and a nonempty subset $A\subseteq X$, the following two conditions are equivalent:
\begin{enumerate}[\rm (1)]
\item $A\in \ir (X)$.
\item For any $U, V\in \mathcal O(X)$, $A\cap U\neq \emptyset$ and  $A\cap V\neq \emptyset$ imply $A\cap U\cap V\neq \emptyset$.
\end{enumerate}
\end{remark}

The following lemma on irreducible sets is well-known.

\begin{lemma}\label{irr-image}
	If $f : X \longrightarrow Y$ is continuous and $A\in\ir (X)$, then $f(A)\in \ir (Y)$.
\end{lemma}

A topological space $X$ is called \emph{sober}, if for any  $A\in\ir_c(X)$, there is a unique point $x\in X$ such that $A=\overline{\{x\}}$. It is straightforward to verify that every sober space is a $d$-space (cf. \cite{GHKLMS-2003}). For simplicity, if a dcpo or a complete lattice $P$ has a sober (resp., non-sober) Scott topology, then we will call $P$ a \emph{sober} (resp., \emph{non-sober}) \emph{dcpo} or a \emph{sober} (resp., \emph{non-sober}) \emph{complete lattice}.

The following two results are well-known.

\begin{proposition}\label{quasicontinuou-domain-Scott-topology-sober} (\cite{GHKLMS-2003, Gierz-Lawson-Stralka-1983}) For a quasicontinuous domain (especially, a continuous domain) $P$, $\Sigma ~\!\! P$ is sober.
\end{proposition}

\begin{proposition}\label{jointly-Scott-continuous-Scott-topology-sober} (\cite{GHKLMS-2003}) If $P$ is a dcpo and a sup semilattice (especially, a complete lattice)
such that the sup operation $\vee :\Sigma~\!\!P\times \Sigma~\!\! P \rightarrow \Sigma~\!\!P, (x, y)\mapsto x\vee y$, is continuous, then $\Sigma~\!\!P$ is a sober space.
\end{proposition}

For a sup semilattice $Q$, by Lemma \ref{Scott-cont-charac} it is easy to verify that the sup operation $\vee :\Sigma~\!\!(P\times P) \rightarrow \Sigma~\!\!P, (x, y)\mapsto x\vee y$, is always continuous, whence by Proposition \ref{jointly-Scott-continuous-Scott-topology-sober} we have the following.

\begin{corollary}\label{Scott-topology-on-product-is-Scott-topology-product-imply-sober} (\cite{GHKLMS-2003}) If $P$ is a dcpo and a sup semilattice (especially, a complete lattice)
such that $\Sigma~\!\!(P\times P)=\Sigma~\!\!P\times \Sigma~\!\! P$, then $\Sigma~\!\!P$ is sober.
\end{corollary}

\begin{remark}\label{examp-Isbell-not-sober}  (1) In \cite{Isbell-1982} (see also \cite{Isbell-1985}), Isbell constructed a complete lattice $L$ whose Scott space is non-sober. By Proposition \ref{jointly-Scott-continuous-Scott-topology-sober}, $\sigma (L\times L)$ is strictly
finer than the product topology $\sigma (L)\times \sigma (L)$, and the sup operation $\vee :\Sigma~\!\!(L\times \Sigma~\!\! L) \rightarrow \Sigma~\!\!L, (x, y)\mapsto x\vee y$, is not continuous.

(2)  In \cite{Isbell-1985}, Isbell showed that there exists a complete
lattice $Q$ for which $\Sigma~\!\!Q$ is sober but the sup operation $\vee :\Sigma~\!\!Q\times \Sigma~\!\! Q \rightarrow \Sigma~\!\!Q, (x, y)\mapsto x\vee y$, is not continuous (whence $\sigma (Q\times Q)\neq \sigma (Q)\times \sigma (Q)$).

(3) In \cite{Hertling-2022}, Hertling presented a complete lattice $Z$ such that the Scott topology on $Z \times Z$ is strictly
finer than the product topology but the sup operation $\vee :\Sigma~\!\!Z\times \Sigma~\!\! Z \rightarrow \Sigma~\!\!Z, (x, y)\mapsto x\vee y$, is continuous.
\end{remark}

For a $T_0$ space $X$, we shall use $\mathord{\mathsf{K}}(X)$ to
denote the set of all nonempty compact saturated subsets of $X$ and endow it with the \emph{Smyth order}, that is, for $K_1,K_2\in \mathord{\mathsf{K}}(X)$, $K_1\sqsubseteq K_2$ if{}f $K_2\subseteq K_1$. Let $S^u(X)=\{\ua x : x\in X\}$. In what follows, the poset $\mk (X)$ is always equipped with the Smyth order. The space $X$ is called \emph{coherent} if the intersection of any two compact saturated sets is compact.

For a subset $G$ of a $T_0$ space $X$, let $\Diamond G=\{K\in \mathsf{K}(X) : K\cap G\neq\emptyset\}$ and $\Box G=\{K\in \mathsf{K}(X) : K\subseteq G\}$. The \emph{upper Vietoris topology} on $\mathsf{K}(X)$ is the topology that has $\{\Box U : U\in \mathcal O(X)\}$ as a base, and the resulting space, denoted by $P_S(X)$, is called the \emph{Smyth power space} or \emph{upper space} of $X$ (cf. \cite{Heckmann-1992, Schalk-1993}). It is easy to verify that the specialization order on $P_S(X)$ is the Smyth order (that is, $\leq_{P_S(X)}=\sqsubseteq$). Clearly, $S^u(X)$, as a subspace of $P_S(X)$, is homeomorphic to $X$.

It is straightforward to verify the following (cf. \cite{Heckmann-1992, Heckmann-Keimel-2013, Schalk-1993}).

\begin{lemma}\label{xi-continuous}
	For a $T_0$ space $X$, then the canonical mapping $\xi_X : X\longrightarrow P_S(X), x\mapsto\uparrow x$, is a topological embedding.
\end{lemma}

\begin{lemma}\label{KX-sup}  For a nonempty family $\{K_i : i\in I\}\subseteq \mk (X)$, $\bigvee_{i\in I} K_i$ exists in $\mk (X)$ if{}f~$\bigcap_{i\in I} K_i\in \mk (X)$. In this case $\bigvee_{i\in I} K_i=\bigcap_{i\in I} K_i$.
\end{lemma}

\begin{proof} Suppose that $\{K_i : i\in I\}\subseteq \mk (X)$ is a nonempty family and $\bigvee_{i\in I} K_i$ exists in $\mk (X)$. Let $K=\bigvee_{i\in I} K_i$. Then $K\subseteq K_i$ for all $i\in I$, and hence $K\subseteq \bigcap_{i\in I} K_i$. For any $x\in \bigcap_{i\in I} K_i$, $\ua x$ is a upper bound of $\{K_i : i\in I\}\subseteq \mk (X)$, whence $K\sqsubseteq \ua x$ or, equivalently, $\ua x \subseteq K$. Therefore, $\bigcap_{i\in I} K_i\subseteq K$. Thus $\bigcap_{i\in I} K_i=K\in \mk (X)$.

Conversely, if $\bigcap_{i\in I} K_i\in \mk (X)$, then $\bigcap_{i\in I} K_i$ is an upper bound of $\{K_i : i\in I\}$ in $\mk (X)$. Let $G\in \mk (X)$ be another upper bound of $\{K_i : i\in I\}$, then $G\subseteq K_i$ for all $i\in I$, and hence $G\subseteq \bigcap_{i\in I} K_i$, that is, $\bigcap_{i\in I} K_i\sqsubseteq G$, proving that $\bigvee_{i\in I} K_i=\bigcap_{i\in I} K_i$.
\end{proof}

For the sobriety of Smyth power spaces, we have the following important result.

\begin{theorem}\label{Smyth-sober} \emph{(\cite{Heckmann-Keimel-2013, Schalk-1993})} For a $T_0$ space $X$, the following conditions are equivalent:
\begin{enumerate}[\rm (1)]
            \item $X$ is sober.
            \item  For any $\mathcal A\subseteq \ir (P_S(X))$ and $U\in \mathcal O(X)$, $\bigcap \mathcal A\subseteq U$ implies $K\subseteq U$ for some $K\in\mathcal A$.
            \item  For any $\mathcal A\subseteq \ir_c (P_S(X))$ and $U\in \mathcal O(X)$, $\bigcap \mathcal A\subseteq U$ implies $K\subseteq U$ for some $K\in\mathcal A$.
            \item $P_S(X)$ is sober.
\end{enumerate}
\end{theorem}

   Rudin's Lemma is a useful tool in topology and plays a crucial role in domain theory (see [7-11, 13, 33]). In \cite{Heckmann-Keimel-2013}, Heckmann and Keimel presented the following topological variant of Rudin's Lemma.

\begin{lemma}\label{topological-Rudin-lemma} \emph{(Topological Rudin Lemma)} Let $X$ be a topological space and $\mathcal{A}$ an
irreducible subset of the Smyth power space $P_S(X)$. Then every closed set $C {\subseteq} X$  that
meets all members of $\mathcal{A}$ contains a minimal irreducible closed subset $A$ that still meets all
members of $\mathcal{A}$.
\end{lemma}

  Applying Lemma \ref{topological-Rudin-lemma} to the Alexandroff topology on a poset $P$, one obtains the original Rudin's Lemma (see \cite{Rudin-1981}).

\begin{corollary}\label{rudin} \emph{(Rudin's Lemma)} Let $P$ be a poset, $C$ a nonempty lower subset of $P$ and $\mathcal F\in \mathbf{Fin} P$ a filtered family with $\mathcal F\subseteq\Diamond C$. Then there exists a directed subset $D$ of $C$ such that $\mathcal F\subseteq \Diamond\da D$.
\end{corollary}

\begin{proposition}\label{X-irreducible-sober-charac}  Let $X$ be a $T_0$ space for which $\ir_c(X)=\{\overline{\{x\}} : x\in X\}\cup\{X\}$. Then the following two conditions are equivalent:
\begin{enumerate}[\rm (1)]
\item $X$ is sober.
\item For any $\mathcal A\in \ir(P_S(X))$, $\bigcap \mathcal A\neq\emptyset$.
\end{enumerate}
\end{proposition}
\begin{proof}  (1) $\Rightarrow$ (2): By Theorem \ref{Smyth-sober}.

(2) $\Rightarrow$ (1): By Lemma \ref{irr-image}, Lemma \ref{xi-continuous} and $X\in \ir_c(X)$, $\xi_X(X)=\{\uparrow x : x\in X\}\in \ir (P_S(X))$, and hence by Condition (2), $\bigcap_{x\in X}\uparrow x\neq\emptyset$ or, equivalently, $X$ has a largest element, say $T$. Then $X=\overline{\{T\}}$. Thus $X$ is sober.

\end{proof}

 A topological space $X$ is called \emph{well-filtered} if it is $T_0$, and for any open set $U$ and filtered family $\mathcal{K}\subseteq \mathord{\mathsf{K}}(X)$, $\bigcap\mathcal{K}{\subseteq} U$ implies $K{\subseteq} U$ for some $K{\in}\mathcal{K}$.

It is well-known that every sober space is well-filtered and every well-filtered space is a $d$-space. Kou \cite{Kou-2001} gave the
first example of a dcpo whose Scott space is well-filtered but
non-sober. Another simpler dcpo whose Scott topology is well-filtered but not sober was presented in \cite{Zhao-Xi-Chen-2019}. In \cite{Jia-2018}, Jia constructed a countable infinite dcpo whose Scott topology is well-filtered but non-sober. It is worth noting that Johnstone \cite{Johnstone-1981} constructed the first example of a dcpo whose Scott
space is non-sober (indeed, it is not well-filtered) and Isbell \cite{Isbell-1982} constructed a complete
lattice whose Scott space is non-sober.

 \begin{proposition}\label{WF-filtered-compact-sets-cap-compact} For a well-filtered space $X$ and a filtered family $\mathcal K\subseteq \mk (X)$, $\bigcap \mathcal K\in \mk (X)$.
\end{proposition}
\begin{proof} Clearly, $\bigcap \mathcal K$ is saturated and $\bigcap \mathcal K\neq\emptyset$ (otherwise, $\bigcap \mathcal K=\emptyset$ implies $K=\emptyset$ for some $K\in \mathcal K$, a contradiction). Now we verify that $\bigcap \mathcal K$ is compact. Let $\{U_i : i\in I\}$ be an open cover of $\bigcap \mathcal K$. As $X$ is well-filtered, there is a $K\in\mathcal K$ such that $K\subseteq \bigcup_{i\in I}U_i$. By the compactness of $K$, there is $J\in I^{(<\omega)}$ such that $K\subseteq \bigcup_{i\in J}U_i$, and hence $\bigcap \mathcal K\subseteq K\subseteq \bigcup_{i\in J}U_i$. Thus $\bigcap \mathcal K\in \mk (X)$.
\end{proof}

\begin{corollary}\label{X-irreducible-wf-space-charac}  Let $X$ be a $T_0$ space for which $\ir_c(X)=\{\overline{\{x\}} : x\in X\}\cup\{X\}$. Then the following two conditions are equivalent:
\begin{enumerate}[\rm (1)]
\item $X$ is well-filtered.
\item For any filtered family $\mathcal{K}\subseteq \mathord{\mathsf{K}}(X)$, $\bigcap\mathcal{K}\neq \emptyset$.
\end{enumerate}
\end{corollary}
\begin{proof}  (1) $\Rightarrow$ (2): By Proposition \ref{WF-filtered-compact-sets-cap-compact}.

(2) $\Rightarrow$ (1): Let $\mathcal{K}$ be a filtered family of compact saturated sets and $U$ an open set with $\bigcap\mathcal{K}{\subseteq} U$. As $\bigcap\mathcal{K}\neq \emptyset$, we have $U\neq\emptyset$. If $K\nsubseteq U$ for each $K\in \mathcal K$, then by Lemma \ref{topological-Rudin-lemma}, $X\setminus U$ contains a minimal irreducible closed subset $A$ that still meets all members of $\mathcal K$. Since $U\neq\emptyset$, we have $A\neq X$ and hence by $\ir_c(X)=\{\overline{\{x\}} : x\in X\}\cup \{X\}$ there is $x\in X$ with $A=\overline{\{x\}}$. It follows that $x\in \bigcap \mathcal K\subseteq U$, which contradicts $x\in A\subseteq X\setminus U$. Thus $K\subseteq U$ for some $K\in \mathcal K$, proving that $X$ is well-filtered.

\end{proof}

In \cite{Xi-Lawson-2017}, Xi and Lawson gave a sufficient condition for a dcpo $P$ with Scott topology to be well-filtered.

\begin{proposition}\label{Lawson-compact-Scott-wf-2} \emph{(\cite{Xi-Lawson-2017})} For a dcpo $P$, if $(P, \lambda (P))$ is upper semicompact (in particular, if $(P, \lambda (P))$ is compact or $P$ is a complete lattice), then $(P, \sigma (P))$ is well-filtered.
\end{proposition}

The following result is well-known.

 \begin{theorem}\label{LC-sober=LC-wf=CC-sober} \emph{(\cite{GHKLMS-2003, Kou-2001})}  For a $T_0$ space $X$, the following conditions are equivalent:
\begin{enumerate}[\rm (1)]
	\item $X$ is locally compact and sober.
	\item $X$ is locally compact and well-filtered.
	\item $X$ is core compact and sober.
\end{enumerate}
\end{theorem}

The above result was improved in \cite{Lawson-Wu-Xi-2020, Xu-Shen-Xi-Zhao-2020-2} by two different methods.

\begin{proposition}\label{CC-wf-sober} \emph{(\cite{Lawson-Wu-Xi-2020, Xu-Shen-Xi-Zhao-2020-2})} Every core compact well-filtered space is sober.
\end{proposition}

\section{Scott topology on product of two posets}\label{Scott-topology-on-product-of-two-dcpos}

\begin{definition}\label{def-cofinal-subset-of-directed-set} Let $P$ be a poset $P$ and $A$ a nonempty set of $P$. A nonempty $B\subseteq A$ is called a \emph{cofinal} subset of $A$, or $B$ is cofinal in $A$, if $A\subseteq \downarrow B$ (equivalently, $\downarrow A=\downarrow B$), that is, for any $a\in A$, there is $b(a)\in B$ such that $a\leq b(a)$.
\end{definition}

For directed subsets of a poset, we have the following folklore result.

\begin{lemma}\label{cofinal-subset-of-directed-set} Let $P$ be a poset and $D$ a directed subset of $P$. If $\{D_1, D_2, ..., D_n\}$ is a finite family of nonempty subsets of $D$ with $D=\bigcup\limits_{i=1}^{n}D_i$, then some $D_m$ ($1\leq m\leq n$) is a cofinal subset of $D$.
\end{lemma}
\begin{proof} Assume, on the contrary, that $D\neq \downarrow D_i$ for all $1\leq i\leq n$. Then for each $i\in \{1, 2, ..., n\}$, there is $d_i\in D$ such that $d_i\not\in \downarrow D_i$. As $D$ is directed, there exists a $d_0\in D$ with $d_0\in \bigcap\limits_{i=1}^{n}\uparrow d_i$. By $D=\bigcup\limits_{i=1}^{n}D_i$, we have $d_0\in D_j$ for some $1\leq j\leq n$ and hence $d_j\in \downarrow d_0\subseteq \downarrow D_j$, a contradiction. Thus there is some $m\in \{1, 2, ..., n\}$ such that $D_m$ is a cofinal subset of $D$.
\end{proof}

\begin{lemma}\label{count-directed-set-sup-count-chain-sup} (\cite{Xu-Shen-Xi-Zhao-2020-1})
Let $P$ be a poset and $D$ a countable directed subset of $P$. Then there exists a countable chain $C\subseteq D$ such that $\da D=\da C$. Hence, $\vee C$ exists and $\vee C=\vee D$ whenever $\vee D$ exists. If $D$ has no largest element, then $C$ can be chosen to be a strictly ascending chain.
\end{lemma}
\begin{proof}
	If $|D|<\omega$, then $D$ contains a largest element $d$, so let $C=\{d\}$, which satisfies the requirement.
	
	Now assume $|D|=\omega$ and let $D=\{d_n:n\in\mathbb{N}\}$. We use induction on $n\in\omega$ to  define $C=\{c_n:n<\omega\}$.
	More precisely, let $c_1=d_1$ and let $c_{n+1}$ ($n\in \mathbb{N}$) be an upper bound of $\{d_{n+1},c_0, c_1,c_2\ldots,c_n\}$ in $D$. It is clear that $C$ is a chain and $\da D=\da C$.

Suppose that $D=\{d_n:n<\omega\}$ is a countable directed and has no largest element. Let $c_1=d_1$. Since $D$ has no largest element, there is $d_{m_1}$ such that $d_{m_1}\nleq c_1$. Let $c_2$ be an upper bound of $\{d_2,c_1, d_{m_1}\}$ in $D$. Then $c_1< c_2$ and $\{d_1, d_2\}\subseteq \downarrow c_2$. We assume generally that for $n\in \mathbb{N}$ we have chosen in $D$ finite elements $c_i$ ($1\leq i\leq n$) such that $ c_1<c_2< ... <c_n$ and $\{d_1, d_2, ..., d_n\}\subseteq \downarrow c_n$. Then as $D$ has no largest element, there is $d_{m_n}$ such that $d_{m_n}\nleq c_n$. Let $c_{n+1}$ be an upper bound of $\{d_{n+1},c_n, d_{m_n}\}$ in $D$. Then $c_n< c_{n+1}$ and $\{d_1, d_2, ..., d_{n+1}\}\subseteq \downarrow c_{n+1}$. So by induction we get a strictly ascending chain $C=\{c_n : n\in\mathbb{N}\}$ satisfying $\da D=\da C$.
\end{proof}

\begin{theorem}\label{Scott-topology-product}  (\cite{GHKLMS-2003})
Let $P$ be a poset. Then the following statements are
equivalent:
\begin{enumerate}[\rm (1)]
\item$\sigma (P)$ is a continuous lattice, that is, $\Sigma~\!\!P$ is core compact.
\item For every poset $S$ one has $\Sigma~\!\!(L\times S)=\Sigma~\!\!L\times \Sigma~\!\!S$, that is, the Scott topology of $L\times S$ is equal to the product of the individual Scott topologies.
\item For every dcpo or complete lattice $S$ one has $\Sigma~\!\!(L\times S)=\Sigma~\!\!L\times \Sigma~\!\!S$.
\item $\Sigma~\!\!(L\times \sigma (L))=\Sigma~\!\!L\times \Sigma~\!\!\sigma (L)$.
\end{enumerate}
\end{theorem}
\begin{proof} It was proved in \cite{GHKLMS-2003} for dcpos (see the proof of \cite[Theorem II-4.13]{GHKLMS-2003}) and the proof is valid for posets.

\end{proof}

\begin{corollary}\label{Scott-topology-product-cor}
Suppose that $P$ is a poset for which $\Sigma~\!\!P$ is locally compact. Then for every poset $S$, $\Sigma~\!\!(P\times S)=\Sigma~\!\!P\times \Sigma~\!\!S$.
\end{corollary}

From Corollary \ref{Scott-topology-on-product-is-Scott-topology-product-imply-sober} and Theorem \ref{Scott-topology-product} we deduce the following.

\begin{corollary}\label{complete-lattice-Scott-core-compact-Scott-topology-sober} (\cite{GHKLMS-2003}) If $P$ is a dcpo and a sup semilattice (especially, a complete lattice) such that $\Sigma~\!\! P$ is core compact (especially, locally compact), then the sup operation $\vee :\Sigma~\!\!P\times \Sigma~\!\! P \rightarrow \Sigma~\!\!P$ is continuous, and hence $\Sigma~\!\!P$ is sober.
\end{corollary}

Matthew de Brecht made the following observation, explained in detail in the blog of Jean
Goubault-Larrecq\footnote{M. de Brecht, On countability: the compact completed sequence,
https://projects.lsv.ens-cachan.fr/topology/?page-id=1852 (2019).}.

\begin{proposition}\label{de-Brecht} Let $P_1$ and $P_2$ be two posets such that the
Scott spaces $\Sigma~\!\!P_1$ and  $\Sigma~\!\!P_2$ are first-countable. Then $\Sigma~\!\! (P_1\times P_2)=\Sigma~\!\!P_1\times \Sigma~\!\!P_2$.
\end{proposition}

From Corollary \ref{Scott-topology-on-product-is-Scott-topology-product-imply-sober} and Proposition \ref{de-Brecht} we deduce the following.

\begin{corollary}\label{cor-de-Brecht}  If $P$ is a dcpo and a sup semilattice (especially, a complete lattice) such that $\Sigma~\!\! P$ is first-countable, then the sup operation $\vee :\Sigma~\!\!P\times \Sigma~\!\! P \rightarrow \Sigma~\!\!P$ is continuous, and hence $\Sigma~\!\!P$ is sober.
\end{corollary}

The theory of compactly generated spaces plays an important role in general and algebraic topology. For Hausdorff spaces compactly generated spaces are those for which the topology of a space is generated by its compact subsets in the sense that a
subset if open if and only if its intersection with each compact subset
is open in the topology of that compact subset. In \cite{Lawson-2019}, Lawson developed a notion of locally compact generation for non-Hausdorff spaces, especially for $T_0$ spaces. In this regard, the case that countably many locally
compact subspaces suffice is particularly interesting.

\begin{definition}\label{def-lc-omega-space} (\cite{Lawson-2019})
A topological space $X$ is called an $\ell c_{\omega}$-\emph{space} if there is a sequence $(K_n)_{n\in \mathbb{N}}$ of locally compact subsets of $X$ satisfying the following conditions:

\smallskip
$(a)$\, $K_{n}\subseteq K_{m}$ whenever $n\leq m$;

\smallskip
$(b)$\, $X=\bigcup_{n}K_{n}$;

\smallskip
$(c)$\, the inclusion $i_{n}:K_{n}\rightarrow K_{n+1}$ is continuous for each $n$;

\smallskip
$(d)$\,  a subset $U$ of $X$ is open if and only if $U\cap K_{n}$ is open in $K_n$ for each $n$.
\end{definition}

Note that we allow the possibility that the topology of $K_n$ may differ from the subspace topology. However, condition (d)
ensures that the topology is at least as fine as the subspace topology.

Now we will show that the product of two $\ell c_{\omega}$-spaces is again a $\ell c_{\omega}$-space.
The proof relies heavily on Wallace's Lemma (see, for example, \cite[Theorem 3.2.10]{Engelking-1989}). We only need the case of product of two spaces.

\begin{lemma}\label{Wallace lemma} (Wallace's Lemma) If $K_i$ is a compact subset of a topological space $X_i$ ($i=1, 2$), then for every open subset $W$ of the product space $X_1\times X_2$ which contains the set $K_1\times K_2$, there exist open sets $U_i$ of $X_i$ ($i=1, 2$) such that $K_1\times K_2\subseteq U_1\times U_2\subseteq W$.
\end{lemma}

\begin{proposition}\label{lc-omega-space-two-product} (\cite{Lawson-2019})
If $X, Y$ are two $\ell c_{\omega}$-spaces, then the product space $X\times Y$ is an $\ell c_{\omega}$-space.
\end{proposition}

\begin{proof} It was proved by Lawson in \cite{Lawson-2019}. For the sake of completeness, we present his proof here.

Since $X$ (resp., $Y$) is an $\ell c_{\omega}$-spaces, there is a sequence $(J_n)_{n\in \mathbb{N}}$ of locally compact subsets of $X$ (resp., a sequence $(K_n)_{n\in \mathbb{N}}$ of locally compact subsets of $Y$) satisfying Conditions (a)-(d) of Definition \ref{def-lc-omega-space}.

For each $n\in \mathbb{N}$, set $L_{n}=J_{n}\times K_{n}$ (equipped with the product topology). Then $(L_n)_{n\in \mathbb{N}}$ is a sequence of locally compact subsets of $X\times Y$ satisfying Conditions (a)-(c) of Definition \ref{def-lc-omega-space}. Now we verify Condition (d). Define $$\tau =\{W\subseteq X\times Y : W\cap L_{n}\mbox{ is open in }~L_{n} \mbox{ for all}~n\}.$$ It is straightforward to verify that $\tau$ is a topology on $X\times Y$. If $U$ is an open subset of $X$ and $V$ is an open subset of $Y$, then $U\cap J_{n}$ is open in $J_{n}$ and $V\cap K_{n}$ is open in $K_{n}$. Hence $(U\times V)\cap L_{n}=(U\cap J_{n})\times (V\cap K_{n})$ is open in $L_{n}=J_{n}\times K_{n}$. It follows that the product open sets $U\times V$ are in $\tau$, and since such sets form a basis for the product topology, $\tau$ contains the product topology.

Conversely suppose that $W\in \tau$ and let $(a,b)\in W$. Since $X\times Y=\bigcup_{n}L_{n}$, it follows that $(a,b)\in J_{n}\times K_{n}=L_{n}$ for some $n$. Since $W\cap L_{n}$ is open, we can pick sets $P_{n}$ open in $J_{n}$ and $Q_{n}$ open in $K_{n}$ such that $(a,b)\in P_{n}\times Q_{n}\subseteq W\cap L_{n}$. Since $J_{n}$ and $K_{n}$ are locally compact, we can pick neighborhoods $U_{n}$ of $a$ in $J_{n}$ and $V_{n}$ of $b$ in $K_{n}$ and sets $C_{n}$ and $D_{n}$ compact in $J_{n}$ and $K_{n}$ resp. with $$(a,b)\in U_{n}\times V_{n}\subseteq C_{n}\times D_{n}\subseteq P_{n}\times Q_{n}\subseteq W\cap L_{n}.$$

By Definition \ref{def-lc-omega-space} (c), $C_{n}$ is compact in $J_{n+1}$ and $D_{n}$ is compact in $K_{n+1}$. By Lemma \ref{Wallace lemma}, there exists $P_{n+1}$ open in $J_{n+1}$ and $Q_{n+1}$ open in $K_{n+1}$ such that $$C_{n}\times D_{n}\subseteq P_{n+1}\times Q_{n+1}\subseteq W\cap J_{n+1}\times K_{n+1}=W\cap L_{n+1}.$$ By local compactness of $J_{n+1}$, for each $x\in C_{n}$ choose a set $U_{x}$ open in $J_{n+1}$ and $C_{x}$ compact in $J_{n+1}$ such that $x\in U_{x}\subseteq C_{x}\subseteq P_{n+1}$. Finitely many of the $U_{x}$ cover $C_{n}$; let $U_{n+1}$ be their union and $C_{n+1}$ be the union of the corresponding $C_{x}$. Then $C_{n+1}$ is compact in $J_{n+1}$ and $C_{n}\subseteq U_{n+1}\subseteq C_{n+1}\subseteq P_{n+1}$. Pick $V_{n+1}$ and $D_{n+1}$ similarly in $K_{n+1}$ such that $D_{n}\subseteq V_{n+1}\subseteq D_{n+1}\subseteq Q_{n+1}$ with $V_{n+1}$ open in $K_{n+1}$ and $D_{n+1}$ compact in $K_{n+1}$. Then $U_{n+1}\times V_{n+1}\subseteq P_{n+1}\times Q_{n+1}\subseteq W$.

We can continue the process of the preceding paragraph by induction obtaining increasing sequences of relative open sets $(U_{n+k})_{k\geq 0}$ and $(V_{n+k})_{k\geq 0}$ and compact sets $(C_{n+k})_{k\geq 0}$ and $(D_{n+k})_{k\geq 0}$ (from Definition \ref{def-lc-omega-space} (c)), where $$C_{n+k}\times D_{n+k}\subseteq U_{n+k+1}\times V_{n+k+1}\subseteq C_{n+k+1}\times D_{n+k+1}\subseteq U_{n+k}\times V_{n+k}\subseteq W$$ for all $k\geq 0$.

We set $U=\bigcup\limits_{k=0}^{\infty}U_{n+k}$ and $V=\bigcup\limits_{k=0}^{\infty}V_{n+k}$. For $j<k$ we note $U_{n+k}\cap J_{n+j}$ is open in $J_{n+j}$ by Definition \ref{def-lc-omega-space} (c). Further, since $(U_{n+k})_{k\geq 0}$ is increasing, we obtain $$U\cap J_{n+j}=\bigcup_{k}(U_{n+k}\cap J_{n+j})=\bigcup_{k\geq j}(U_{n+k}\cap U_{n+j})=U_{n+j}.$$ Each member of the righthand union is open in $J_{n+j}$ (again Definition \ref{def-lc-omega-space} (c)), and hence the union is. For $m<n$, $U\cap J_{m}=U\cap (J_{m}\cap J_{n})=(U\cap J_{n})\cap J_{m}$, which is open in $J_{m}$ by Definition \ref{def-lc-omega-space} (c) and the preceding calculation. Since $(X,\{J_{n}\})$ is an $\ell c_{\omega}$-space, we conclude that $U$ is an open subset of $X$. Similarly, $V$ is an open set of $Y$. From the preceding paragraph we have $U\times V\subseteq W$. We have thus found an open set $U\times V$ of $(a,b)$ in the product topology contained in $W$. It follows that the product topology contains the $\tau$-topology, which completes the proof.
\end{proof}

\begin{definition}\label{def-sub-dcpo} A \emph{sub-dcpo} $Q$ of a dcpo $P$ is a subset of $P$ with the restricted order such that if $D$ is a directed subset of $Q$, then $\vee_P D\in Q$. More generally, a \emph{d-sub-poset} $S$ of a poset $T$ is a subset of $T$ with the restricted order such that if $D$ is a directed subset of $S$ for which $\vee_T D$ exists, then $\vee_T D\in S$.
\end{definition}

\begin{definition}\label{def-lc-omega-dcpo}  (\cite{Lawson-2019})
A poset (resp., dcpo) $P$ is said to
be an $\ell c_{\omega}$-\emph{poset} (resp., $\ell c_{\omega}$-\emph{dcpo}) if there exists a sequence $(K_n)_{n\in \mathbb{N}}$ of $d$-sub-posets (resp., sub-dcpos) satisfying the following conditions:

\smallskip
$(a)$\, $\Sigma~\!\!K_{n}$ is locally compact for all $n\in \mathbb{N}$;

\smallskip

$(b)$\, $K_{n}\subseteq K_{m}$ whenever $n\leq m$;

\smallskip
$(c)$\, $P=\bigcup_{n}K_{n}$;

\smallskip
$(d)$\, a subset $U$ of $P$ is Scott open if and only if $U\cap K_{n}$ is Scott open in $K_n$ for each $n$.
\end{definition}

\begin{remark}\label{Lawson-condition-c}
Since $K_{n}$ and $K_{n+1}$ are $d$-sub-posets of $P$, the inclusion of $K_{n}$ into $K_{n+1}$ is automatically (Scott) continuous, so condition (c)
of Definition \ref{def-lc-omega-space} is automatically satisfied. We note that similarly the
inclusion of $K_n$ into $P$ is also (Scott) continuous.
\end{remark}

\begin{theorem}\label{Scott-topology-on-product-of-two-lc-omega-poset} (\cite{Lawson-2019})
Let $P, Q$ be two $\ell c_{\omega}$-posets. Then $\Sigma~\!\! (P\times Q)=\Sigma~\!\!P\times \Sigma~\!\!Q$.
\end{theorem}

\begin{proof} The following proof is taken from \cite{Lawson-2019}. Let $(J_n)_{n\in \mathbb{N}}$ resp. $(K_n)_{n\in \mathbb{N}}$ be a sequence of $d$-sub-posets satisfying Conditions (a)-(d) of Definition \ref{def-lc-omega-dcpo} for $P$ resp. $Q$.

It follows from Proposition \ref{lc-omega-space-two-product} that the product topology $\sigma (P)\times \sigma (Q)$ is equal to the topology $$\tau =\{W\subseteq P\times Q:\forall n, W\cap (J_{n}\times K_{n}) \mbox{ is open in } J_{n}\times K_n\}.$$
Since the Scott topology of the product always contains the product of the Scott topologies, it suffices to show that a Scott open set of $P\times Q$ is $\tau$-open.

Let $W$ be a Scott-open subset of $P\times Q$. As noted in Remark the inclusions $j_n : J_n\rightarrow P$ and $k_n:K_{n}\rightarrow Q$ are Scott continuous. It follows directly that $j_{n}\times k_{n}:J_{n}\times K_{n}\rightarrow P\times Q$ preserves directed sups and hence is Scott continuous, where both domain and codomain are given their Scott topologies. Thus $W\cap (J_{n}\times K_{n})$ is Scott-open in $J_{n}\times K_{n}$. By Corollary \ref{Scott-topology-product-cor}, the local compactness of $\Sigma~\!\!K_{n}$ implies that the Scott topology on $J_{n}\times K_{n}$ is equal to the product of the Scott topologies on $J_n$ and $K_n$ respectively. Thus $W$ is $\tau$-open.
\end{proof}

\begin{proposition}\label{lc-omega-poset-two-product}
If $P, Q$ are two $\ell f_{\omega}$-poset (resp., $\ell f_{\omega}$-dcpo), then the product $P\times Q$ is an $\ell f_{\omega}$-poset (resp., $\ell c_{\omega}$-dcpo).
\end{proposition}
\begin{proof} As $P$ and $Q$ are two $\ell f_{\omega}$-poset (resp., $\ell _{\omega}$-dcpo), there exists a sequence $(J_n)_{n\in \mathbb{N}}$ of $d$-sub-posets (resp., sub-dcpos) of $P$ satisfying Conditions (a)-(d) of Definition \ref{def-lc-omega-dcpo} for $P$, and a sequence $(K_n)_{n\in \mathbb{N}}$ of $d$-sub-posets (resp., sub-dcpos) of $Q$ satisfying Conditions (a)-(d) of Definition \ref{def-lc-omega-dcpo} for $Q$. For each $n\in \mathbb{N}$, let $H_n=J_n\times K_n$. Then $(H_n)_{n\in \mathbb{N}}$ is a sequence of $d$-sub-posets (resp., sub-dcpos) of $P\times Q$. Since the product space of two locally compact spaces is locally compact, by Theorem \ref{Scott-topology-product} and Theorem \ref{Scott-topology-on-product-of-two-lc-omega-poset}, it is easy to verify that $(H_n)_{n\in \mathbb{N}}$ satisfies Conditions (a)-(d) of Definition \ref{def-lc-omega-dcpo} for $P\times Q$. Thus $P\times Q$ is an $\ell c_{\omega}$-poset (resp., $\ell c_{\omega}$-dcpo).
\end{proof}

 \begin{definition}\label{def-LF-space} A $T_0$ space $X$ is called  \emph{locally finite} if for each $x\in X$ and each open neighborhood $U$ of $x$, there is a finite subset $F$ of $X$ such that $x\in\ii\,\ua F=\ua F\subseteq U$ or, equivalently, there is an $F\in X^{(<\omega)}$ such that $\ua F\in \mathcal O(X)$ and $x\in\ua F\subseteq U$.
\end{definition}

Similar to the concept of $\ell c_{\omega}$-posets, we give the following.

\begin{definition}\label{def-lf-omega-dcpo}
A poset (resp., dcpo) $P$ is said to be an $\ell f_{\omega}$-\emph{poset} (resp., $\ell f_{\omega}$-\emph{dcpo}) if there is a sequence $(K_n)_{n\in \mathbb{N}}$ of $d$-sub-posets (resp., sub-dcpos) satisfying the following conditions:

\smallskip
$(a)$\, $\Sigma~\!\!K_{n}$ is locally finite;

\smallskip

$(b)$\, $K_{n}\subseteq K_{m}$ whenever $n\leq m$;

\smallskip
$(c)$\, $P=\bigcup_{n\in \mathbb{N}}K_{n}$;

\smallskip
$(d)$\, a subset $U$ of $P$ is Scott open if and only if $U\cap K_{n}$ is Scott open in $K_n$ for each $n$.
\end{definition}

Clearly, if $P$ is an $\ell f_{\omega}$-poset (resp., $\ell f_{\omega}$-dcpo), then $P$ is an $\ell c_{\omega}$-poset (resp., $\ell c_{\omega}$-dcpo). The following example shows that the reversion does not hold in general.

\begin{example}\label{lc-omega-dcpo-not-imply-ln-omega-dcpo} Let $W$ be the unital compact topological semilattice, constructed in \cite[VI-4]{GHKLMS-2003}, which has no basis of subsemilattices. It was proved in \cite{Xu-2016-1} (see also \cite{Erne-1985}) that $W$ is a meet continuous lattice and $\sigma(W)$ is a continuous lattice, but $W$ is not a continuous lattice. Hence $W$ is not quasicontinuous. Then by Theorem \ref{LC-sober=LC-wf=CC-sober} and Proposition \ref{Lawson-compact-Scott-wf-2}, $\Sigma~\!\!W$ is locally compact and hence $W$ is an $\ell c_{\omega}$-dcpo. By Lemma \ref{quasicontinuou-domain=Scott-topology-LHC}, $\Sigma~\!\!W$ is not locally hypercompact. It is not difficult to show that $W$ is not an $\ell f_{\omega}$-dcpo.
\end{example}

Similar to Proposition \ref{lc-omega-poset-two-product}, we have the following.

\begin{proposition}\label{lf-omega-poset-two-product}
If $P, Q$ are two $\ell f_{\omega}$-poset (resp., $\ell f_{\omega}$-dcpo), then the product $P\times Q$ is an $\ell f_{\omega}$-poset (resp., $\ell c_{\omega}$-dcpo).
\end{proposition}
\begin{proof} Based on the fact that the product space of two locally finite spaces is locally finite, the proof is completely similar to that of Proposition \ref{lc-omega-poset-two-product}.
\end{proof}

For a poset $P$, $I\subseteq P$ is called an \emph{ideal} of $P$ if $I$ is a directed lower subset of $P$. Let $\mathrm{Id} (P)$ be the poset (with the order of set inclusion) of all ideals of $P$. The ideals $\da x$ are called \emph{principal ideals}. $I\in \mathrm{Id} (P)$ is called a \emph{non-trivial ideal} if $I$ is not a principal ideal, namely, $I\neq \downarrow x$ for all $x\in P$. Let $\mathrm{Id}_p(P)=\{\da x : x\in P\}$ and $\mathrm{Id}_{\neg p}(P)=\mathrm{Id}(P)\setminus \mathrm{Id}_p(P)$, namely, the family of all non-trivial ideals of $P$.

\begin{remark}\label{rem-no-principal-ideal=directed-set-has-largest-element} For a poset $P$, consider the following conditions:
\begin{enumerate}[\rm (1)]
\item There is no non-trivial ideal in $P$ (i.e., $\mathrm{Id}(P)=\{\da_{P} x : x\in P\}$).
\item Every directed subset of $P$ has a largest element.
\item Every element of $P$ is compact.
\item $\sigma (P)=\alpha (P)$.
\end{enumerate}
Then (1) $\Leftrightarrow$ (2) $\Rightarrow$ (3) $\Leftrightarrow$ (4), and the four conditions are equivalent if $P$ is a dcpo.
\end{remark}

Now we give one of the main results of this paper.

\begin{theorem}\label{c-poset-is-lf-omega-poset} Let $P$ be a poset (resp., a dcpo) satisfying the following conditions:
\begin{enumerate}[\rm (i)]
\item $\mathrm{Id}_{\neg p} (P)=\{I_n : n\in \mathbb{N}\}$ is countable;
\item for each $n\in \mathbb{N}$, there is a countably strict chain $c_{(n, 1)}<c_{(n, 2)}< ... < c_{(n, m)}<c_{(n, m+1)}<...$ such that $I_n=\da
\{c_{(n, m)} : m\in \mathbb{N}\}$;
\item there is a countable family $\{T_n : n\in \mathbb{N}\}$ of subsets of $P$ such that $P\setminus \bigcup_{n\in\mathbb{N}} H_n=\bigcup_{n\in \mathbb{N}}T_n$ and $\mathrm{Id}(T_l)=\{\da_{T_l} x : x\in T_l\}$ (that is, $\mathrm{Id}_{\neg p} (T_l)=\emptyset$) for all $l\in \mathbb{N}$, where $H_n=\{c_{(n, m)} : m\in \mathbb{N}\}\cup \{\bigvee\limits_{m\in \mathbb{N}}c_{(n, m)}\}$ if $\bigvee\limits_{m\in \mathbb{N}}c_{(n, m)}$ exists in $P$ and $H_n=\{c_{(n, m)} : m\in \mathbb{N}\}$ if $\bigvee\limits_{m\in \mathbb{N}}c_{(n, m)}$ does not exist in $P$.
\end{enumerate}
Then $P$ is an $\ell f_{\omega}$-poset (resp., an $\ell f_{\omega}$-dcpo).
\end{theorem}
\begin{proof} We only give a proof for the case that $P$ is a poset. When $P$ is a dcpo, the proof is clearly valid.

For each $n\in \mathbb{N}$, let $C_n=\{c_{(n, m)} : m\in \mathbb{N}\}$ and let $K_n=\bigcup\limits_{i=1}^{n}H_i\cup\bigcup\limits_{i=1}^{n} T_i$. Now we show that $(K_n)_{n\in \mathbb{N}}$ is a sequence of $d$-sub-posets of $P$ satisfying Conditions (a)-(d) of Definition \ref{def-lf-omega-dcpo}.

{\bf Claim 1:} For each $n\in \mathbb{N}$, let $S_n=\bigcup\limits_{i=1}^{n} T_i$. Then every ideal of $S_n$ is principal, that is, $\mathrm{Id} (S_n)=\{\da_{S_n} x : x\in S_n\}$.

Let $E\in \mathcal D(S_n)$. Then by Lemma \ref{cofinal-subset-of-directed-set} there is $1\leq j\leq n$ such that $E\cap T_j$ is a cofinal subset of $E$. Since every ideal of $T_j$ is principal by assumption (Condition (iii)), $E\cap T_j$ (as a directed subset of $T_j$) has a largest element $e$. Clearly,  $e$ is the largest element of $E$ and hence $\da_{S_n} E=\da_{S_n} e$. Thus $\mathrm{Id} (S_n)=\{\da_{S_n} x : x\in S_n\}$.

{\bf Claim 2:} For each $n\in \mathbb{N}$, $K_n$ is a $d$-sub-poset of $P$.

 For $D\in \mathcal D(K_n)$ with $\vee D$ existing in $P$, if $D$ has no largest element, then by Claim 1, $D\cap S_n$ is not a cofinal subset of $D$. Hence by Lemma \ref{cofinal-subset-of-directed-set}, there is $1\leq i\leq n$ such that $D\cap H_i$ is a cofinal subset of $D$ (equivalently, $D\cap C_i$ is a cofinal subset of $D$ since $D$ has no largest element) and hence $\vee (D\cap H_i)=\vee D$. As $D$ has no largest element, $D\cap H_i$ is a countable infinite sub-chain of $H_i$. It follows that $\vee H_i$ exists in $P$ (equivalently, $\vee C_i$ exists in $P$) and $\vee D=\vee (D\cap H_i)=\vee H_i=\vee C_i\in H_i\subseteq K_n$. Thus $K_n$ is a $d$-sub-poset of $P$.

{\bf Claim 3:} $K_{n}\subseteq K_{m}$ whenever $n\leq m$ and $P=\bigcup_{n\in \mathbb{N}}K_{n}$.

Clearly, $K_n\subseteq K_{n+1}$ for each $n\in \mathbb{N}$ and $\bigcup_{n\in\mathbb{N}}K_{n}=\bigcup\limits_{n\in \mathbb{N}}(\bigcup\limits_{i=1}^{n}H_i\cup\bigcup\limits_{i=1}^{n} T_i)=\bigcup\limits_{n\in \mathbb{N}}H_n\cup \bigcup\limits_{n\in \mathbb{N}}T_n=P$.

{\bf Claim 4:} For each $n\in \mathbb{N}$, $\Sigma~\!\!K_{n}$ is locally finite.

Let $x\in U\in \sigma(K_n)$. We denote by $\mathbb{N}_U$ the finite set $\{i\in \mathbb{N} : 1\leq i\leq n, U\cap C_i\neq\emptyset\}$. Then for each $j\in \mathbb{N}_U$, $c_{(j, k)}\in U$ for some $k\in \mathbb{N}$. Let $m_j=\mathrm{min}\{m\in\mathbb{N} : c_{(j, m)}\in U\}$ and let $F_U=\{x\}\cup\{c_{(j, m_j)} : j\in\mathbb{N}_U\}$. Then $F\in (K_n)^{(<\omega)}$ and $x\in \ua_{K_n}F\subseteq U$. Now we show that $\ua_{K_n}F$ is Scott open in $K_n$. Suppose that $E\in \mathcal D(K_n)$ with $\vee_{K_n} E$ existing and $\vee_{K_n} E\in \ua_{K_n}F$. If $E$ has a largest element $e$, then $e=\vee_{K_n} E\in E\cap \ua_{K_n}F$. Now we assume that $E$ has no largest element. Then carrying out a proof similar to that of Claim 2, we know that there is some $l\in\{1, 2, ..., n\}$ such that $E\cap C_l$ is a cofinal subset of $C_l$ and hence $\vee_{K_n} (E\cap C_l)=\vee_{K_n} E\in \ua_{K_n}F\subseteq U\in \sigma(K_n)$. It follows that there is $e^{\prime}=c_{(l, m)}\in E\cap C_l\cap U$. Then $e^{\prime}\in \ua_{K_n}c_{(l, m_l)}\subseteq \ua_{K_n}F$, completing the proof of the Scott-openness of $\ua_{K_n}F$ in $K_n$.

{\bf Claim 5:} For any $U\subseteq P$, $U\in \sigma(P)$ iff $U\cap K_{n}\in \sigma(K_n)$ for each $n$.

For each $n\in \mathbb{N}$, as $K_n$ is a $d$-sub-poset of $P$, we clearly have that $U\in \sigma(P)$ implies $U\cap K_{n}\in \sigma(K_n)$. Conversely, suppose that $U\cap K_{n}\in \sigma(K_n)$ for each $n\in \mathbb{N}$. Now we show that $U$ is Scott open in $P$. First, we verify that $U=\ua U$. Let $y\geq x\in U$. By $P=\bigcup_{n\in \mathbb{N}}K_{n}$, there is $m\in \mathbb{N}$ such that $x, y\in K_m$. Then $y\geq x\in U\cap K_m\in \sigma (K_m)$. Hence $y\in U\cap K_m\subseteq U$. For $D\in \mathcal D(P)$ with $\vee D$ existing and $\vee D\in U$, if $D$ has a largest element $d^{\ast}$, then $d^{\ast}=\vee D\in U$. If $D$ has no largest element, then $\da D\in \mathrm{Id}_{\neg p} (P)$, whence $\da D=\da C_n$ for some $n\in \mathbb{N}$. Since $K_n$ is a $d$-sub-poset of $P$, we have that $\vee_{K_n}C_n=\vee C_n=\vee D\in U\cap K_n\in\sigma(K_n)$ by assumption. It follows that $c_{(n, l)}\in U\cap K_n$ for some $l\in \mathbb{N}$. By $\da D=\da C_n$, there is $d_l\in D$ with $c_{(n, l)}\leq d_l$, and consequently, $d_l\in \ua U=U$. Thus $U\in \sigma (P)$.

By Claims 1-4, $P$ is an $\ell f_{\omega}$-poset.
\end{proof}

\begin{definition}\label{def-c-poset} A poset (resp., dcpo) $P$ is called a $c$-\emph{poset} (resp., a $c$-\emph{dcpo}) if $P$ satisfies Conditions (i)-(iii) of Theorem \ref{c-poset-is-lf-omega-poset}.
\end{definition}

By Theorem \ref{Scott-topology-on-product-of-two-lc-omega-poset}, Proposition \ref{lf-omega-poset-two-product} and Theorem \ref{c-poset-is-lf-omega-poset}, we get the following corollary.

\begin{corollary}\label{product-of-finite-c-poset-is-lc-omega-poset} Let $\{P_i : i=1, 2, 3, ..., n\}$ ($n\in\mathbb{N}$) be a finite family of $c$-posets (resp., $c$-dcpos). Then the product $\prod\limits_{i=1}^{n}P_i$ is an $\ell f_{\omega}$-poset (resp., an $\ell f_{\omega}$-dcpo) and $\Sigma~\!\!\prod\limits_{i=1}^{n}P_i=\prod\limits_{i=1}^{n}\Sigma~\!\!P_i$.
\end{corollary}

\begin{proposition}\label{poset-P-Id-P-countable-imply-P-is-c-poset} Let $P$ be a poset (resp., a dcpo) for which $\mathrm{Id} (P)$ is countable. Then $P$ is a $c$-poset (resp., a $c$-dcpo) and hence an $\ell f_{\omega}$-poset (resp., an $\ell f_{\omega}$-dcpo).
\end{proposition}
\begin{proof} Since $|\mathrm{Id} (P)|\leq \omega$ and $\{\downarrow x : x\in P\}\subseteq \mathrm{Id} (P)$, we have that $\mathrm{Id}_{\neg p} (P)$ is countable and $|P|=|\{\downarrow x : x\in P\}|\leq |\mathrm{Id} (P)|\leq \omega$ (i.e.,  $P$ is countable). If $P$ is finite (equivalently, $\mathrm{Id} (P)$ is finite), then $x\ll x$ for all $x\in P$, and consequently, $\sigma (P)=\{\ua A : A\subseteq P\}\cup\{\emptyset\}$. Hence $P$ is a $c$-poset. Now suppose that $P$ is countable infinite. Then $\mathrm{Id} (P)$ is also countable infinite. Let $\mathrm{Id}_{\neg p} (P)=\{I_n : n\in \mathbb{N}\}$. For each $n\in \mathbb{N}$, by Lemma \ref{count-directed-set-sup-count-chain-sup}, there exists a countable strictly chain $c_{(n,1)}< c_{(n, 2)}< c_{(n, 3)}< ... < c_{(n, m)}< c_{(n, m+1)}<...$ in $I_n$ such that $I=\da \{c_{(n, m)} : m\in \mathbb{N}\}$. For each $n\in \mathbb{N}$, let $C_n=\{c_{(n, m)} : m\in \mathbb{N}\}$ and let

$$H_n=
	\begin{cases}
	C_n\cup \{\vee C_n\}, & \mathrm{~if~} \vee C_n \mathrm{~exists~in~} P \\
	C_n,& \mathrm{~if~} \vee C_n \mathrm{~does~not~exist~in~} P .
	\end{cases}$$

\noindent Since $P$ is countable, the subset $P\setminus \bigcup_{n\in \mathbb{N}}H_n$ is countable. Let $P\setminus \bigcup_{n\in \mathbb{N}}H_n=\{y_m : m\in \mathbb{N}\}$ and let $T_n=\{y_j : j=1, 2, ..., n\}$ (or $T_n=\{y_n\})$ for each $n\in \mathbb{N}$. Clearly, Conditions (i)-(iii) of Theorem \ref{c-poset-is-lf-omega-poset} are satisfied. Hence $P$ is a $c$-poset and hence an $\ell f_{\omega}$-poset by Theorem \ref{c-poset-is-lf-omega-poset}.
\end{proof}

\begin{corollary}\label{c-poset-is-lf-omega-poset-cor} Let $\{P_i : i=1, 2, 3, ..., n\}$ ($n\in\mathbb{N}$) be a finite family of posets (resp., dcpos) for which all $\mathrm{Id} (P_i)$ are countable. Then $\mathrm{Id} (\prod\limits_{i=1}^{n}P_i)$ is countable. So $\prod\limits_{i=1}^{n}P_i$ is a $c$-poset (resp., a $c$-dcpo) and hence an $\ell f_{\omega}$-poset (resp., an $\ell f_{\omega}$-dcpo).
\end{corollary}

\begin{proof} It is easy to verify that $\mathrm{Id} (\prod\limits_{i=1}^{n}P_i)=\prod\limits_{i=1}^{n}\mathrm{Id} (P_i)$. Hence $\mathrm{Id} (\prod\limits_{i=1}^{n}P_i)$ is countable. By Proposition \ref{poset-P-Id-P-countable-imply-P-is-c-poset}, $\prod\limits_{i=1}^{n} P_i$ is a $c$-poset (resp., a $c$-dcpo) and hence an $\ell f_{\omega}$-poset (resp., an $\ell f_{\omega}$-dcpo).
\end{proof}

As an immediate corollary of Corollary \ref{product-of-finite-c-poset-is-lc-omega-poset} and Proposition \ref{poset-P-Id-P-countable-imply-P-is-c-poset},  we get the following.

\begin{corollary}\label{Scott-topology-on-product-of-two-special-countable-dcpos} (\cite{Miao-Xi-Li-Zhao-2022})
Let $\{P_i :  i=1, 2, 3, ..., n\}$ ($n\in\mathbb{N}$) be a finite family of posets. If all $\mathrm{Id} (P_i)$ are countable, then $\Sigma~\!\!\prod\limits_{i=1}^{n}P_i=\prod\limits_{i=1}^{n}\Sigma~\!\!P_i$.
\end{corollary}

\begin{example}\label{examp-Scott-topology-on-product-of-two-Johnstone-dcpos}  Let $\mathbb{J}=\mathbb{N}\times (\mathbb{N}\cup \{\infty\})$ with ordering defined by $(j, k)\leq (m, n)$ if{}f $j = m$ and $k \leq n$, or $n =\infty$ and $k\leq m$ (see (see Figure 1). $\mathbb{J}$ is a well-known dcpo constructed by Johnstone in \cite{Johnstone-1981}. Clearly, $\mathbb{J}_{max}=\{(n, \infty) : n\in \mathbb{N}\}$ is the set of all maximal elements of $\mathbb{J}$, and for each $n\in \mathbb{N}$, $(n, 1)<(n, 2)< ... < (n, m) < (n, m+1) < ...$ is a countably strict chain with $(n, \infty)=\bigvee\limits_{m\in \mathbb{N}}(n, m)$.
\begin{figure}[ht]
	\centering
	\includegraphics[height=4.5cm,width=4.5cm]{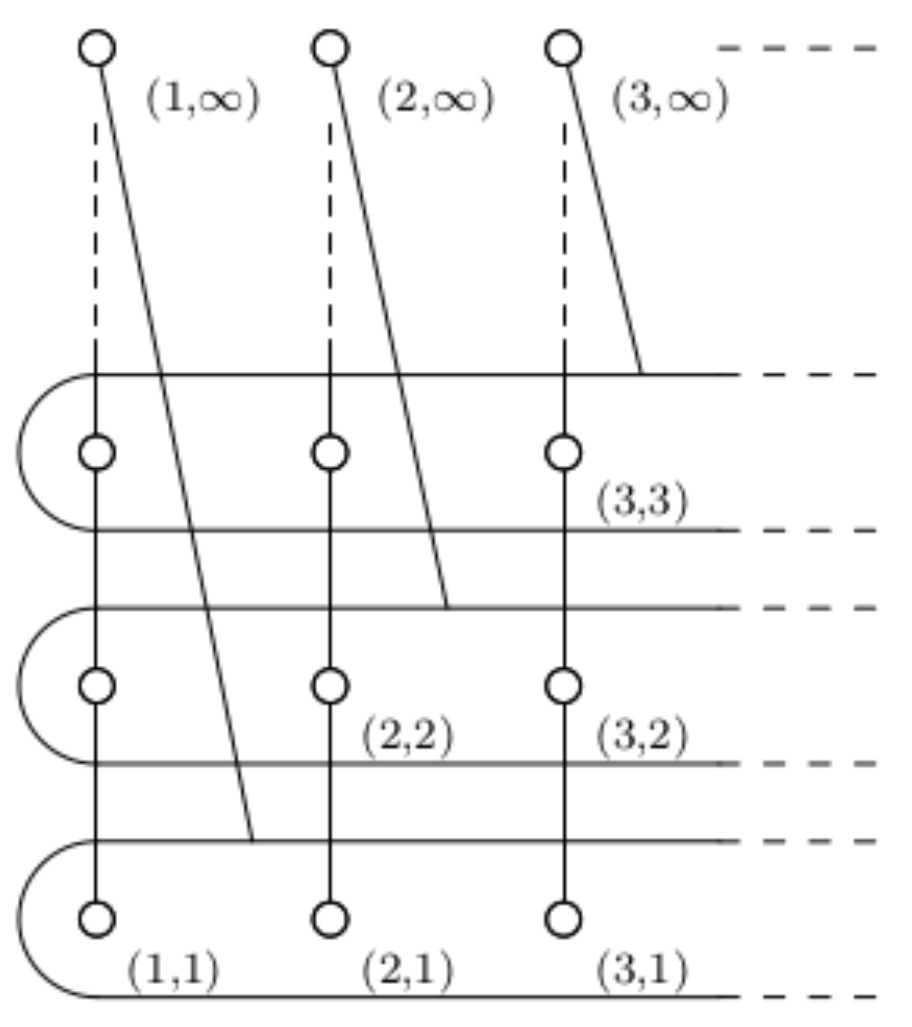}
	\caption{Johnstone's dcpo $\mathbb{J}$}
\end{figure}

The following three conclusions about $\Sigma~\!\mathbb{J}$ are known (see, for example, \cite[Example 3.1]{Lu-Li-2017} and \cite[Lemma 3.1]{Miao-Li-Zhao-2021}):
\begin{enumerate}[\rm (i)]
\item $\ir_c (\Sigma~\!\mathbb{J})=\{\overline{\{x\}}=\da_{\mathbb{J}} x : x\in \mathbb{J}\}\cup \{\mathbb{J}\}$.
\item $\mathsf{K}(\Sigma~\!\mathbb{J})=(2^{\mathbb{J}_{max}} \setminus \{\emptyset\})\bigcup \mathbf{Fin}~\!\mathbb{J}$.
\item $\Sigma~\!\mathbb{J}$ is coherent by (ii).
\item $\Sigma~\!\mathbb{J}$ is not well-filtered and hence non-sober (cf. \cite[Exercise 8.3.9]{Goubault-2013}).

Let $\mathcal G=\{\mathbb{J}\setminus F : F\in (\mathbb{J}_{max})^{(<\omega)}\}$. Then by (ii), $\mathcal G\subseteq \mathsf{K}(\Sigma~\!\mathbb{J}_\top)$ is a filtered family and $\bigcap\mathcal{G}=\bigcap_{F\in (\mathbb{J}_{max})^{(<\omega)}} (\mathbb{J}\setminus F)=\mathbb{J}_{max}\setminus \bigcup (\mathbb{J}_{max})^{(<\omega)}=\emptyset$, but $\mathbb{J}\setminus F=\emptyset$ for no $F\in (\mathbb{J}_{max})^{(<\omega)}$. Hence $\Sigma~\!\mathbb{J}$ is not well-filtered.
\end{enumerate}

Let $X$ be any set ($X$ can be the empty set, or a countable set, or a uncountably infinite set). $X$ can be considered as a poset equipped with the discrete order. Let $P=\mathbb{J}\cup X$ with ordering defined by
$$x\leq y \mbox{ if and only if }
\begin{cases}
	x\leq y\mbox{ in } \mathbb{J}, \mathrm{if}~ x,y\in \mathbb{J}, \mathrm{or}\\
	x=y, x, y\in X.
	\end{cases}$$

\noindent Clearly, when $|X|>\omega$, $\mathrm{Id}(P)$ is not countable. Indeed, $|\mathrm{Id}(P)|=|X|>\omega$. Now we show that $P$ is a $c$-dcpo.
For each $n\in \mathbb{N}$, let $C_n=\{n\}\times \mathbb{N}$ and $H_n=C_n\cup \{\vee C_n=(n, \infty)\}$. It is straightforward to verify that $\mathrm{Id}_{\neg p}(P)=\{\da C_n: n\in\mathbb{N}\}$. Clearly, $P\setminus \bigcup_{n\in\mathbb{N}} H_n=X$ and $\mathrm{Id}(X)=\{\da_{X} x=\{x\} : x\in X\}$. Therefore, Conditions (i)-(iii) of Theorem \ref{c-poset-is-lf-omega-poset} are satisfied and hence $P$ is a $c$-dcpo. By Corollary \ref{product-of-finite-c-poset-is-lc-omega-poset}, $\Sigma~\!\!P\times P=\Sigma~\!\!P\times \Sigma~\!\!P$. In particular, when $X=\emptyset$ we have that $\mathrm{Id} (\mathbb{J})=\{\downarrow x : x\in \mathbb{J}\}\cup \{\downarrow C_n : n\in \mathbb{N}\}$ is countable and hence $\Sigma~\!\mathbb{J}\times \mathbb{J}=\Sigma~\!\mathbb{J}\times \Sigma~\!\mathbb{J}$ by Corollary \ref{Scott-topology-on-product-of-two-special-countable-dcpos}.
\end{example}

The following example shows that there is a countable complete lattice $L$ such that $\Sigma~\!\! L$ is not first-countable but $\Sigma~\!\!\sigma(L)$ is second-countable (cf. \cite{Chen-Kou-Lyu-2022, Hertling-2022, Xu-Shen-Xi-Zhao-2020-1}).
	
\begin{example}\label{exam-a-countable-complete-lattice-non-first-countable}
	Let $L=\{\bot\}\cup(\mn \times \mn)\cup\{\top\}$ and define a partial order $\leq$ on $L$ as follows (see Figure 2):
	\begin{itemize}
		\item [(i)] $\forall (n,m)\in \mn\times\mn$, $\bot\leq (n,m) \leq\top$;
		\item [(ii)] $\forall (n_1,m_1), (n_2,m_2)\in\mn\times\mn$, $(n_1,m_1)\leq(n_2,m_2)$ iff $n_1=n_2$ and $m_1\leq m_2$.
	\end{itemize}
\begin{figure}[ht]
	\centering
	\includegraphics[height=4.5cm,width=4.5cm]{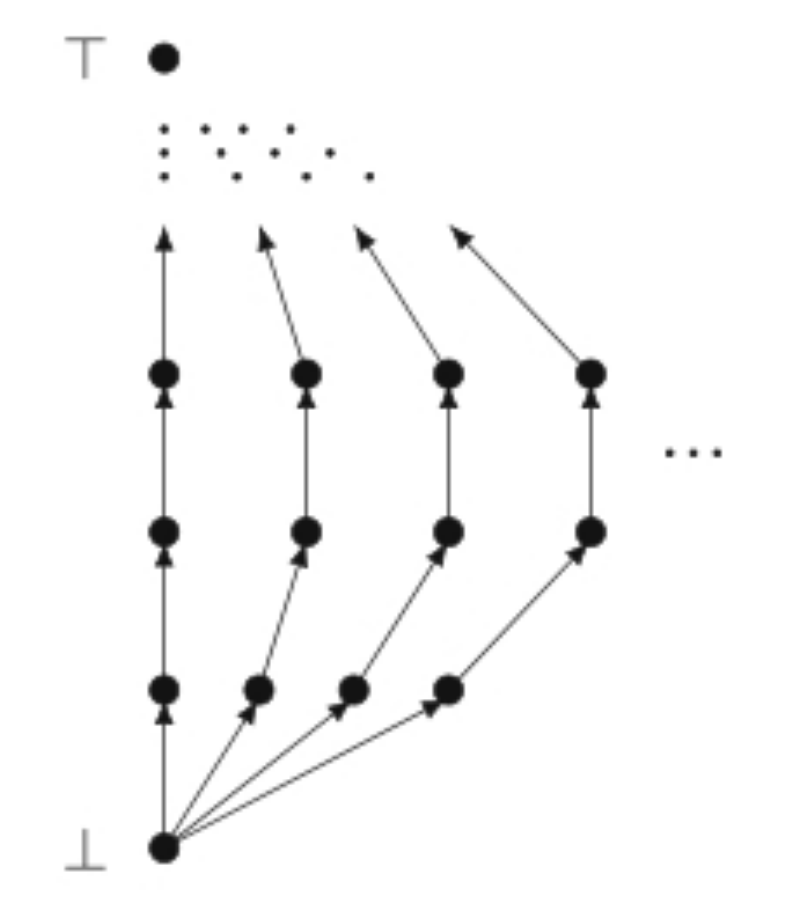}
	\caption{Complete lattice $L$ in Example 4.28.}
\end{figure}

\begin{enumerate}[\rm (1)]
\item It was shown in \cite[Example 4.8.]{Xu-Shen-Xi-Zhao-2020-1} that $\Sigma~\!\!L$ is not first-countable.
\item It is easy to see that $\mathrm{Id}(L)$ is countable. Hence by Proposition \ref{poset-P-Id-P-countable-imply-P-is-c-poset} and Corollary \ref{Scott-topology-on-product-of-two-special-countable-dcpos}, $L$ is a $c$-dcpo and an $\ell f_{\omega}$-dcpo, and $\Sigma~\!\! (L\times L)=\Sigma~\!\!L\times \Sigma~\!\!L$ (see also \cite[Lemma 4.12]{Hertling-2022}). So by Corollary \ref{Scott-topology-on-product-is-Scott-topology-product-imply-sober} $\Sigma~\!\!L$ is sober.
\item Hertling proved in \cite[Lemma 4.13 and Corollary 4.16]{Hertling-2022} that $\Sigma ~\!\!\sigma (L)$ is second-countable (i.e., $\sigma (L)$ has a countable base), whence by Proposition \ref{de-Brecht}, $\Sigma~\!\! (\sigma (L)\times \sigma (L))=\Sigma~\!\!\sigma (L)\times \Sigma~\!\!\sigma (L)$.
\item One can easily check that $\Sigma ~\!\!L$ is not core compact (indeed, for $U, V\in \sigma (L)$, $U\ll_{\sigma (L)}V$ iff $U=\emptyset$ or $U=V=L$). Therefore, by Theorem \ref{Scott-topology-product}, $\Sigma~\!\!(L\times \sigma (L))\neq\Sigma~\!\!L\times \Sigma~\!\!\sigma (L)$ (see also \cite[Proposition 4.6]{Hertling-2022}).

\end{enumerate}
\end{example}

\begin{example}\label{exam-Scott-topology-on-product-of-two-Jia-dcpos}

Let $\mathcal{L}=\mathbb{N}\times \mathbb{N}\times (\mathbb{N}\cup \{\infty\})$, where $\mathbb{N}$ is the set of natural numbers with the usual order. Define an order $\leq$ on $\mathcal L$ as follows:

$(i_1, j_1, k_1)\leq (i_2, j_2, k_2)$ if and only if:

(1) $i_1=i_2, j_1=j_2, k_1\leq k_2\leq \infty$; or

(2) $i_2=i_1+1, k_1\leq j_2, k_2=\infty$.

 $\mathcal L$ is a known dcpo constructed by Jia in \cite[Example 2.6.1]{Jia-2018}. It can be easily depicted as in Figure 3 taken from \cite{Jia-2018}.

\begin{figure}[ht]
	\centering
	\includegraphics[height=5cm,width=12cm]{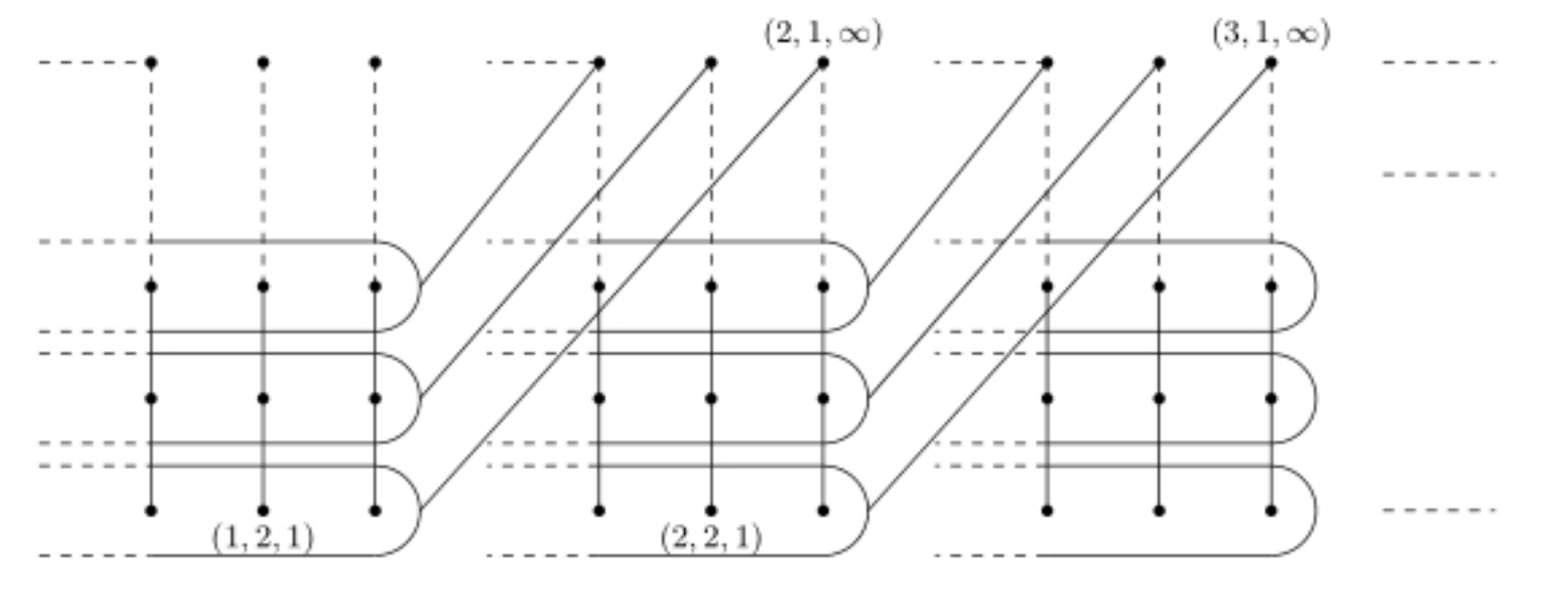}
	\caption{Jias dcpo $\mathcal L$}
\end{figure}

Let $\mathcal L^{\infty}=\{(i, j, \infty) : (i, j)\in \mathbb{N}\times\mathbb{N}\}$ (the set of all maximal elements of $\mathcal L$).

\begin{enumerate}[\rm (1)]
\item It was shown in \cite[Example 2.6.1]{Jia-2018} that $\Sigma~\!\!\mathcal{L}$ is well-filtered but non-sober.

\item  $\Sigma~\!\!L$ is not coherent.

By \cite[Claim 2.6.3]{Jia-2018}, $\ua (1, 1, 1)\cap \ua (1, 2, 1)=\{(2, m+1, \infty) : m\in \mathbb{N}\}$ is not Scott compact. Thus $\Sigma~\!\!\mathcal{L}$ is not coherent.

\item Suppose that $D$ is an infinite directed subset of $\mathcal L$. Then there is a unique $(i_D, j_D, \infty)\in \mathcal L^{\infty}$ such that $(i_D, j_D, \infty)$ is a largest element of $D$ or the following three conditions are satisfied:
\begin{enumerate}[\rm (i)]
\item $(i_D, j_D, \infty)\not\in D$,
\item $D\subseteq \{(i_D, j_D, l) : l\in \mathbb{N}\}$, and
\item  $(i_D, j_D, \infty)=\bigvee_{\mathcal L} D$.
\end{enumerate}
{\bf Proof.} If there is $(i_0, j_0, \infty)\in D\cap \mathcal L^{\infty}$, then for each $d=(i_d, j_d, l_d)\in D$, there is $d^*=(i_{d^*}, j_{d^*}, l_{d^*})\in D$ such that $(i_0, j_0, \infty)\leq d^*=(i_{d^*}, j_{d^*}, l_{d^*})$ and $d=(i_d, j_d, l_d)\leq d^*=(i_{d^*}, j_{d^*}, l_{d^*})$, whence $l_{d^*}=\infty, i_{d^*}=i_0, j_{d^*}=j_0$ (i.e., $d^*=(i_0, j_0, \infty)$) and $d\leq d^*=(i_0, j_0, \infty)$. Hence $(i_D, j_D, \infty)=(i_0, j_0, \infty)$ is the (unique) largest element of $D$.

Now suppose that $D\cap \mathcal L^{\infty}=\emptyset$, that is, $D\subseteq \mathbb{N}\times \mathbb{N}\times \mathbb{N}$. Select a $d_1=(i_{d_1}, j_{d_1}, l_{d_1})\in D$. Then for each $d=(i_d, j_d, l_d)\in D$, by the directedness of $D$, there is $d^{\prime}=(i_{d^{\prime}}, j_{d^{\prime}}, l_{d^{\prime}})\in D$ such that $d_1=(i_{d_1}, j_{d_1}, l_{d_1})\leq d^{\prime}=(i_{d^{\prime}}, j_{d^{\prime}}, l_{d^{\prime}})$ and $d=(i_d, j_d, l_d)\leq d^{\prime}=(i_{d^{\prime}}, j_{d^{\prime}}, l_{d^{\prime}})$. Hence $i_{d^{\prime}}=i_d=i_{d_0}, j_{d^{\prime}}=j_d=j_{d_0}$ and $l_{d_0}\leq l_{d^{\prime}}, l_d\leq l_{d^{\prime}}$. Let $i_D=i_{d_1}$ and $j_D=j_{d_1}$. Then $D\subseteq \{(i_D, j_D, l) : l\in \mathbb{N}\}$. Clearly, $(i_D, j_D, \infty)$ is the unique element of $\mathcal L^{\infty}$ satisfying Conditions (i)-(iii).

For any $D\in \mathcal D(\mathcal L)$, if $D$ has a largest element $s$, then $s=\vee D$; if $D$ has no largest element, then $D$ is infinite directed subset of $\mathcal L$, whence there is a unique $(i, j, \infty)\in \mathcal L^{\infty}$ such that Conditions (i)-(iii) are satisfied. Then $(i, j, \infty)=\bigvee_{\mathcal L} D$.

\item $\mathcal L$ is a dcpo and $\mathrm{Id} (\mathcal L)$ is countable by (3).

\item $\Sigma~\!\! (\mathcal L\times \mathcal L)=\Sigma~\!\! \mathcal L\times \Sigma~\!\!\mathcal L$ by (4) and Corollary \ref{Scott-topology-on-product-of-two-special-countable-dcpos}.

\end{enumerate}

\end{example}

\section{Property R and sobriety of Scott topology on dcpos}\label{Sec-Sobriety-of-Scott-topology-on-dcpos}

The following concept was first introduced in \cite[Definition 10.2.11]{Xu-2016-2} (see also \cite{Wen-Xu-2018}).

\begin{definition}\label{def-property-R}  (\cite{Xu-2016-2}) A $T_0$ space $X$ is said to have \emph{property R} if for any family $\{\mathord{\uparrow}F_{i}:i\in I\}\subseteq \mathbf{Fin}~P$ and any $U\in \mathcal O(X)$, $\bigcap_{i\in I}\mathord{\uparrow}F_{i}\subseteq U$ implies  $\bigcap_{i\in I_{0}}\mathord{\uparrow}F_{i}\subseteq U$ for some $\exists I_{0}\in I^{(<\omega)}$. For a poset $P$, when $\Sigma~\!\!P$ has property R, we will simply say that $P$ has property R.
\end{definition}

\begin{proposition}\label{basic-property-of-property-R} For a poset $P$, consider the following conditions:
\begin{enumerate}[\rm (1)]
\item $(P,\lambda(P))$ is compact.
\item $(P,\lambda(P))$ is upper semicompact.
 \item $P$ has property R.
 \item For any filtered family $\{\mathord{\uparrow}F_{i}:i\in I\}\subseteq \mathbf{Fin}~P$ and any $U\in \sigma (P)$, $\bigcap_{i\in I}\mathord{\uparrow}F_{i}\subseteq U$ implies $\mathord{\uparrow}F_{j}\subseteq U$ for some $j\in I$.
 \item $P$ is a dcpo.
  \end{enumerate}
  \noindent Then (1) $\Rightarrow$ (2) $\Rightarrow$ (3) $\Rightarrow$ (4) $\Rightarrow$ (5).
\end{proposition}
\begin{proof} (1) $\Rightarrow$ (2): Trivial.

(2) $\Rightarrow$ (3): Suppose that $\{\mathord{\uparrow}F_{j}:j\in J\}\subseteq \mathbf{Fin}~P$ and $U\in \sigma (P)$ with $\bigcap_{j\in J}\mathord{\uparrow}F_{j}\subseteq U$. Select an $j_0\in J$. Then $\uparrow F_{j_0}\subseteq U\cup \bigcup_{j\in J\setminus\{j_0\}}(P\setminus \mathord{\uparrow}F_{j})$. As $(P,\lambda(P))$ is upper semicompact and $F_{j_0}$ is finite, $\uparrow F_{j_0}$ is Lawson compact, whence there is $J_{0}\in (J\setminus\{j_0\})^{(<\omega)}$ such that $\uparrow F_{j_0}\subseteq U\cup \bigcup_{j\in J_0}(P\setminus \mathord{\uparrow}F_{j})$ or, equivalently, $\bigcap_{j\in J_{0}\cup\{j_o\}}\mathord{\uparrow}F_{j}\subseteq U$. Thus $P$ has property R.

(3) $\Rightarrow$ (4): For a filtered family $\{\mathord{\uparrow}F_{i}:i\in I\}\subseteq \mathbf{Fin}~P$ and $U\in \sigma (P)$ with $\bigcap_{i\in I}\mathord{\uparrow}F_{i}\subseteq U$, since $P$ has property R, there is $I_{0}\in I^{(<\omega)}$ such that $\bigcap_{i\in I_{0}}\mathord{\uparrow}F_{i}\subseteq U$. As $\{\mathord{\uparrow}F_{i}:i\in I\}$ is filtered, there exists a $j\in I$ such that $F_j\subseteq  \bigcap_{i\in I}\mathord{\uparrow}F_{i}$. Hence $\mathord{\uparrow}F_{j}\subseteq U$.

(4) $\Rightarrow$ (5): Let $D\in \mathcal D(P)$. Then $\{\uparrow d : d\in D\}$ is a filtered family. We show that $\overline{D}\cap \bigcap_{d\in D}\{\uparrow d : d\in D\}\neq\emptyset$. Assume, on the contrary, that $\overline{D}\cap \bigcap_{d\in D}\{\uparrow d : d\in D\}=\emptyset$, then $\bigcap_{d\in D}\{\uparrow d : d\in D\}\subseteq P\setminus \overline{D}\in \sigma (P)$. By (3), $\uparrow d\subseteq P\setminus \overline{D}$ for some $d\in D$, a contradiction. So $\overline{D}\cap \bigcap_{d\in D}\{\uparrow d : d\in D\}\neq\emptyset$. Select an $x\in \overline{D}\cap \bigcap_{d\in D}\{\uparrow d : d\in D\}$. It is easy to verify that $x=\vee D$. Therefore, $P$ is a dcpo.
\end{proof}

\begin{theorem}\label{property-R-charac} For a poset $P$, the following conditions are equivalent:
\begin{enumerate}[\rm (1)]
\item $P$ has property R.
\item For any $\{x_{i}:i\in I\}\subseteq P$ and any $U\in \sigma (P)$ with $\bigcap_{i\in I}\mathord{\uparrow}x_{i}\subseteq U$, there is $I_{0}\in I^{(<\omega)}$ such that $\bigcap_{i\in I_{0}}\mathord{\uparrow}x_{i}\subseteq U$.
\item The compact saturated subsets of $(P, \omega(P))$ are exactly the closed subsets of $\Sigma~\!\!P$.
\item Every Scott closet subset of $P$ is compact in $\Sigma~\!\!P$.
\end{enumerate}
\end{theorem}
\begin{proof} (1)$\Rightarrow$(2): Trivial.

(2)$\Rightarrow$(3): Suppose that $C$ is a compact saturated subset of $(P,\omega(P))$, then $C=\downarrow_P C$ (note that $\leq_{\omega (P)}=(\leq_P)^{op}$). By Lemma \ref{upper-semiclosed-closed}, $C\in\mathcal C(\Sigma~\!\!P)$. Conversely, assume that $C$ is a Scott closed subset of $P$ and $\{x_{i}:i\in I\}\subseteq P$ with $C\subseteq \bigcup_{i\in I}(P\setminus \mathord{\uparrow}x_{i})=P\setminus \bigcap_{i\in I}\mathord{\uparrow}x_{i}$. Then $\bigcap_{i\in I}\mathord{\uparrow}x_{i}\subseteq P\setminus C\in \sigma(P)$. By (2), there is $I_{0}\in I^{(<\omega)}$ such that $\bigcap_{i\in I_{0}}\mathord{\uparrow}x_{i}\subseteq P\setminus C$ or, equivalently, $C\subseteq \bigcup_{i\in I_{0}}(P\setminus \mathord{\uparrow}x_{i})$. Thus $C$ is compact in $(P,\omega(P))$ by Alexander's Subbase Lemma (see, eg., \cite[Proposition I-3.22]{GHKLMS-2003}).

(3)$\Rightarrow$(4): Trivial.

(4)$\Rightarrow$(1): Suppose that $\{\mathord{\uparrow}F_{i}:i\in I\}\subseteq \mathbf{Fin}~P$ and $U\in \sigma (P)$ with $\bigcap_{i\in I}\mathord{\uparrow}F_{i}\subseteq U$. Then $P\setminus U\subseteq \bigcup_{i\in I}(P\setminus \mathord{\uparrow}F_{i})$. By (4), $P\setminus U$ is compact in $(P,\omega(P))$, and hence there is $I_{0}\in I^{(<\omega)}$ such that $P\setminus U\subseteq \bigcup_{i\in I_{0}}(P\setminus \mathord{\uparrow}F_{i})$, that is, $\bigcap_{i\in I_{0}}\mathord{\uparrow}F_{i}\subseteq U$. Thus $P$ has property R.
\end{proof}

\begin{proposition}\label{Scott-compact-Lawson-compact} Let $P$ be a poset. If $\Sigma~\!\!P$ is compact and $\bigcap_{x\in F}\uparrow x$ is compact in $\Sigma~\!\!P$ for each $F\in P^{(<\omega)}$, then the following two conditions are equivalent:
\begin{enumerate}[\rm (1)]
\item $(P,\lambda(P))$ is compact.
\item $P$ has property R.
\end{enumerate}
\end{proposition}
\begin{proof} (1)$\Rightarrow$(2): By Proposition \ref{basic-property-of-property-R}.

(2) $\Rightarrow$ (1): Let $\{U_{i}\in \sigma(P) : i\in I\}\cup \{P\setminus \mathord{\uparrow}F_{j}:\mathord{\uparrow}F_{j}\in \mathbf{Fin}~P, j\in J\}$ be a Lawson open cover of $P$. Then $P=\bigcup_{i\in I}U_{i}\cup \bigcup_{j\in J}(P\setminus \mathord{\uparrow}F_{j})=\bigcup_{i\in I}U_{i}\cup (P\setminus \bigcap_{j\in J}\mathord{\uparrow}F_{j})$.

{\bf Case 1:} $\bigcup_{j\in J}(P\setminus \mathord{\uparrow}F_{j})=\emptyset$, that is, $P=\bigcup_{i\in I}U_{i}$.

Then by the compactness of $\Sigma~\!\!P$, there is $I_{0}\in I^{(<\omega)}$ such that $P=\mathord{\uparrow}F=\bigcup_{i\in I_{0}}U_{i}$.

{\bf Case 2:} $\bigcup_{j\in J}(P\setminus \mathord{\uparrow}F_{j})\not=\emptyset$.

Then $\bigcap_{j\in J}\mathord{\uparrow}F_{j}\subseteq \bigcup_{i\in I}U_{i}\in \sigma(P)$. By the property R of $P$, there exists $J_{0}\in J^{(<\omega)}$ such that $\bigcap_{j\in J_{0}}\mathord{\uparrow}F_{j}\subseteq \bigcup_{i\in I}U_{i}$. By assumption, $\bigcap_{j\in J_{0}}\mathord{\uparrow}F_{j}=\bigcup \{\bigcap_{j\in J_{0}}\mathord{\uparrow}\varphi(j):\varphi\in \prod_{j\in J_{0}}F_{j}\}$ is compact in $(P,\sigma(P))$, whence $\bigcap_{j\in J_{0}}\mathord{\uparrow}F_{j}\subseteq \bigcup_{i\in I_{0}}U_{i}$ for some $I_{0}\in I^{(<\omega)}$. So $P=\bigcup_{i\in I_{0}}U_{i}\cup \bigcup_{j\in J_{0}}(P\setminus \mathord{\uparrow}F_{j})$.

Thus $(P,\lambda(P))$ is compact.
\end{proof}

Clearly, if $(P,\lambda(P))$ is compact, then $(P,\sigma(P))$ is compact and for each $F=\{x_1, x_2, ..., x_n\}\in P^{(<\omega)}$, $\mathord{\uparrow}x_{1}\cap \mathord{\uparrow}x_{2}\cap \ldots \cap \mathord{\uparrow}x_{n}$ is a closed subset of $(P,\lambda(P))$, whence it is compact in $(P,\lambda(P))$, and consequently, compact in $(P,\sigma(P))$. So Proposition \ref{Scott-compact-Lawson-compact} can be restated as the following one.

\begin{proposition}\label{Scott-compact-Lawson-compact+1} For a poset, the following two conditions are equivalent:
 \begin{enumerate}[\rm (1)]
\item $(P,\lambda(P))$ is compact.
\item $\Sigma~\!\!P$ is compact, $\mathord{\uparrow}x_{1}\cap \mathord{\uparrow}x_{2}\cap \ldots \cap \mathord{\uparrow}x_{n}$ is Scott compact (i.e., compact in $\Sigma~\!\!P$) for each $\{x_1, x_2, ..., x_n\}\in P^{(<\omega)}$ and $P$ has property R.
\end{enumerate}
\end{proposition}

The following concept was introduced by Lawson, Wu and Xi in \cite{Lawson-Wu-Xi-2020}.

\begin{definition}\label{def-Lawson-Omega-star-compact} (\cite{Lawson-Wu-Xi-2020}) A $T_0$ space $X$ is said to be $\Omega^{\ast}$-\emph{compact} if every closed subset of $X$ is compact in $(X, \omega(X))$.
\end{definition}

As $\{X\setminus \ua F : F\in \mathbf{Fin} X\}$ is a base of $\omega (X)$, we clearly have the following result.

\begin{proposition}\label{property-R-equiv-Omega-star-compact} Let $X$ be a $T_0$ space. Then the following two conditions are equivalent:
\begin{enumerate}[\rm (1)]
\item $X$ has property R.
\item $X$ is $\Omega^{\ast}$-compact.
\end{enumerate}
\end{proposition}

\begin{remark}\label{Xu-book-property-R} Proposition \ref{basic-property-of-property-R}, Theorem \ref{property-R-charac}, Proposition \ref{Scott-compact-Lawson-compact} and Proposition \ref{Scott-compact-Lawson-compact+1} were first given in \cite{Xu-2016-2}. Since the book \cite{Xu-2016-2} was written in Chinese, we present them and their proofs here. See \cite{Xu-2016-2} for more discussions about property R.
\end{remark}

\begin{lemma}\label{lem-Jia-Jung-Li} (\cite{Jia-Jung-2016})  Let $P$ be a dcpo for which $\Sigma~\!\!P$ is well-filtered. Then the following three conditions are equivalent:
\begin{enumerate}[\rm (1)]
\item $\Sigma~\!\!P$ is coherent.
\item $\mathord{\uparrow}x_{1}\cap \mathord{\uparrow}x_{2}\cap \ldots \cap \mathord{\uparrow}x_{n}$ is Scott compact for all finite nonempty set $\{x_1, x_2, ..., x_n\}$ of $P$.
\item $\uparrow x \cap \uparrow y$ is Scott compact for all $x, y\in P$.
\end{enumerate}
\end{lemma}

\begin{proposition}\label{cor-Jia-Jung-Li}  Let $P$ be a dcpo for which $\Sigma~\!\!P$ is well-filtered and coherent. Then $P$ has property R.
\end{proposition}
\begin{proof} Suppose that $\{x_{i}:i\in I\}\subseteq P$ and $U\in \sigma (P)$ with $\bigcap_{i\in I}\mathord{\uparrow}x_{i}\subseteq U$. For each $J\in I^{(<\omega)}$, let $K_{J}=\bigcap_{i\in J}\mathord{\uparrow} x_{i}$. If there is $J_0\in I^{(<\omega)}$ such that $K_{J_0}=\emptyset$, then $\bigcap_{i\in J_0}\mathord{\uparrow} x_{i}\subseteq U$. Now we assume that $K_{J}\neq\emptyset$ for all $J\in I^{(<\omega)}$. Then by Lemma \ref{lem-Jia-Jung-Li}, $\{K_J : J\in I^{(<\omega)}\}\subseteq \mathsf{K}(\Sigma~\!\!P)$ is a filtered family and $\bigcap\limits_{J\in I^{(<\omega)}}K_J= \bigcap_{i\in I}\mathord{\uparrow}x_{i}\subseteq U$. Hence by the well-filteredness of $\Sigma~\!\!P$, $\bigcap_{i\in J^{\prime}}\mathord{\uparrow} x_{i}=K_{J^{\prime}}\subseteq U$  for some $J^{\prime}\in I^{(<\omega)}$. Thus $P$ has property R.
\end{proof}

In what follows, for a poset $P$, let $\omega^{*}(P)$ denote the family of all closed subsets of $(P, \omega(P))$, namely, $\omega^{*}(P)=\{\emptyset, P\}\cup\{\bigcap \mathcal F : \mathcal F\subseteq \mathbf{Fin}P\}$.

\begin{lemma}\label{lem-Q-m-continuous} Let $P$ be a poset and $Q=(\omega^{*}(P),\supseteq)$ (i.e., the order on $Q$ is the reverse inclusion order). Then $m : P\times P\rightarrow Q$, $(x,y)\mapsto \mathord{\uparrow}x\cap\mathord{\uparrow} y$, is Scott continuous.
\end{lemma}
\begin{proof} Suppose that $\{(x_{d},y_{d}):d\in D\}\in\mathcal D (P\times P)$ for which $\bigvee_{d\in D}(x_d, y_d)$ exists in $P\times P$. Then $\{x_{d}:d\in D\}\in \mathcal D(P)$, $\{y_{d}:d\in D\}\in\mathcal D(P)$, and $\bigvee_{d\in D}x_d$ and $\bigvee_{d\in D} y_d$ exist in $P$. Clearly, $m(\bigvee_{d\in D}(x_{d},y_{d}))=m((\bigvee_{d\in D}x_{d},\bigvee_{d\in D}y_{d}))=(\mathord{\uparrow} \bigvee_{d\in D}x_{d})\cap(\mathord{\uparrow}\bigvee_{d\in D}y_{d})=(\bigcap_{d\in D}\mathord{\uparrow} x_{d})\cap(\bigcap_{d\in D}\mathord{\uparrow} y_{d})=\bigcap_{d\in D}(\mathord{\uparrow} x_{d}\cap \mathord{\uparrow} y_{d})=\bigcap_{d\in D}m(x_{d},y_{d})=\bigvee_{Q}\{m(x_{d},y_{d}):d\in D\}$. By Lemma \ref{Scott-cont-charac}, $m:\Sigma~\!\! (P\times P) \rightarrow \Sigma~\!\!Q$ is continuous.
\end{proof}

\begin{corollary}\label{cor-Q-m-continuous} Let $P$ be a poset and $Q=(\omega^{*}(P),\supseteq)$. If $\Sigma~\!\!(P\times P)=\Sigma~\!\!P\times\Sigma~\!\!P$, then $m: \Sigma~\!\!P\times \Sigma~\!\!P\rightarrow \Sigma~\!\!Q$, $m(x,y)=\uparrow x\cap\uparrow y$, is continuous.
\end{corollary}

By Theorem \ref{Scott-topology-product} and Corollary \ref{cor-Q-m-continuous}, we have the following.

\begin{corollary}\label{cor-core-compact-Q-m-continuous}  Let $P$ be a poset and $Q=(\omega^{*}(P),\supseteq)$. If $\Sigma~\!\!P$ is core compact (especially, locally compact), then $m:\Sigma~\!\!P\times \Sigma~\!\!P\rightarrow \Sigma~\!\!Q$, $m(x,y)=\uparrow x\cap\uparrow y$, is continuous.
\end{corollary}

Using the Scott topology and the lower topology, we give the following useful characterization of property R.

\begin{lemma}\label{lem-charat-property-R-using-Scott-topology} Let $P$ be a poset and $Q=(\omega^{*}(P),\supseteq)$. Then the following two conditions are equivalent:
\begin{enumerate}[\rm (1)]
 \item $P$ has property R.
\item For any $U\in\sigma(P)$, $\Phi(U)=\{C\in\omega^{*}(P):C\subseteq U\}\in\sigma(Q)$.
\end{enumerate}
\end{lemma}
\begin{proof} (1)$\Rightarrow$(2): Clearly, $\Phi(U)$ is an upper set of $Q$. Suppose that $\mathcal G=\{C_{d}:d\in D\}\in\mathcal D(Q)$ with $\bigvee_{Q}\mathcal{G}=\bigcap_{d\in D}C_{d}\in\Phi(U)$, i.e., $\bigcap_{d\in D}C_{d}\subseteq U$. Then $P\setminus U\subseteq \bigcup_{d\in D}(P\setminus C_{d})$ and $\{P\setminus C_{d}:d\in D\}\subseteq \omega(P)$ is directed. By the property R of $P$ and
 Theorem \ref{property-R-charac}, $P\setminus U$ is compact in $(P,\omega(P))$, whence there is $d\in D$ such that $P\setminus U\subseteq P\setminus C_{d}$ or, equivalently, $C_{d}\in \Phi(U)$. Thus $\Phi(U)\in\sigma(Q)$.

(2)$\Rightarrow$(1): Let $\{\mathord{\uparrow} F_{i}:i\in I\}\subseteq\mathbf{Fin}~P$ and $U\in\sigma(P)$ with $\bigcap_{i\in I}\mathord{\uparrow} F_{i}\subseteq U$. For each $T\in I^{(<\omega)}$, let $C_{T}=\bigcap_{i\in T}\mathord{\uparrow} F_{i}$. Then $\{C_{T}:T\in I^{(<\omega)}\}\in \mathcal D(Q)$ and $\bigcap_{T\in I^{(<\omega)}}C_{T}=\bigcap_{i\in I}\mathord{\uparrow} F_{i}\subseteq U$, that is, $\bigvee_{Q}\{C_T : T\in I^{(<\omega)}\}=\bigcap_{T\in I^{(<\omega)}}C_{T}\in\Phi(U)$. By $\Phi(U)\in\sigma(Q)$, there exists $S\in I^{(<\omega)}$ such that $C_{S}=\bigcap_{i\in S}\mathord{\uparrow} F_{i}\in\Phi(U)$, i.e., $C_{S}=\bigcap_{i\in S}\mathord{\uparrow} F_{i}\subseteq U$. So $P$ has property R.
\end{proof}

Now we give another main result of this paper.

\begin{theorem}\label{property-R-Scott-sober} Let $P$ be a poset satisfying property R and $\Sigma~\!\!(P\times P)=\Sigma~\!\!P\times\Sigma~\!\!P$. Then $\Sigma~\!\!P$ is sober.
\end{theorem}

\begin{proof} First, by Proposition \ref{basic-property-of-property-R}, $P$ is a dcpo. Suppose that $A\in \ir_c(\Sigma~\!\!P)$. We show that $A$ is directed. Assume, on the contrary, that $A$ is not directed. Then there exist $b, c\in A$ such that $\mathord{\uparrow} b\cap\mathord{\uparrow} c\cap A=\emptyset$, whence $\mathord{\uparrow} b\cap\mathord{\uparrow} c\subseteq P\setminus A\in \sigma(P)$. Let $Q=(\omega^{*}(P),\supseteq)$ and $U=P\setminus A$. Then by Lemma \ref{lem-charat-property-R-using-Scott-topology}, $\Phi(U)=\{C\in\omega^{*}(P):C\subseteq U\}\in \sigma(Q)$. By $\Sigma~\!\!(P\times P)=\Sigma~\!\!P\times\Sigma~\!\!P$ and Corollary \ref{cor-Q-m-continuous}, $m:\Sigma~\!\!P\times\Sigma~\!\!P \rightarrow (Q,\sigma(Q))$, $m(x,y)=\mathord{\uparrow} x\cap\mathord{\uparrow} y$, is continuous, and hence by $m(b,c)=\mathord{\uparrow} b\cap\mathord{\uparrow} c\in\Phi(U)\in\sigma(Q)$, there exist $V,W\in\sigma(P)$ such that $(b,c)\in V\times W\subseteq m^{-1}(\Phi(U))$. As $A\in \ir_c(\Sigma~\!\!P)$, $A\cap V\neq\emptyset$ and $A\cap W\neq\emptyset$, we have $A\cap V\cap W\neq\emptyset$. Select a $z\in A\cap V\cap W$. Then $m(z,z)=\mathord{\uparrow} z\cap\mathord{\uparrow} z=\mathord{\uparrow} z\in m^{-1}(\Phi(U))$, i.e., $\mathord{\uparrow} z\subseteq U=P\setminus A$, a contradiction. So $A$ is directed and hence $\vee A\in A$, and consequently, $A=\mathord{\downarrow} \vee A=\cl_{\sigma(P)}\{\vee A\}$. Thus $\Sigma~\!\!P$ is sober.
\end{proof}

\begin{theorem}\label{property-R-Scott-sober-cor1} Let $P$ be a poset satisfying one of the following conditions:
\begin{enumerate}[\rm (i)]
\item $\mathrm{Id} (P)$ is countable.
\item $\Sigma~\!\!P$ is a $c$-space.
 \item $P$ is an $\ell f_{\omega}$-poset.
 \item $P$ is an $\ell c_{\omega}$-poset.
 \item $\Sigma~\!\!(P\times P)=\Sigma~\!\!P\times\Sigma~\!\!P$.
\end{enumerate}
Suppose that $P$ also satisfies one of the following conditions:
 \begin{enumerate}[\rm (a)]
\item $(P, \lambda (P))$ is compact.
\item $(P, \lambda (P))$ is upper semicompact.
\item $\Sigma~\!\!P$ is well-filtered and coherent.
 \item $P$ has property R.
\end{enumerate}
 \noindent Then $\Sigma~\!\!P$ is sober.
\end{theorem}
\begin{proof} By Theorem \ref{Scott-topology-on-product-of-two-lc-omega-poset}, Theorem \ref{c-poset-is-lf-omega-poset} and Proposition \ref{poset-P-Id-P-countable-imply-P-is-c-poset}, we have that (i) $\Rightarrow$ (ii) $\Rightarrow$ (iii) $\Rightarrow$ (iv) $\Rightarrow$ (v).
From Proposition \ref{basic-property-of-property-R} and Proposition \ref{cor-Jia-Jung-Li} we deduce that (a) $\Rightarrow$ (b) $\Rightarrow$ (d), and (c) $\Rightarrow$ (d). Hence by Theorem \ref{property-R-Scott-sober} we get Theorem \ref{property-R-Scott-sober-cor1}.

\end{proof}

\begin{remark}\label{property-R-Scott-sober-cor4} The result that Condition (i) and Condition (c) in Theorem \ref{property-R-Scott-sober-cor1} together imply the sobriety of $\Sigma~\!\!P$ was proved in \cite{Miao-Xi-Li-Zhao-2022} using a different method.
\end{remark}

For a countable dcpo $P$, Example \ref{examp-Scott-topology-on-product-of-two-Johnstone-dcpos} and Example \ref{exam-Scott-topology-on-product-of-two-Jia-dcpos} show that if $P$ only satisfies one of Conditions (i)-(v) of Theorem \ref{property-R-Scott-sober-cor1} (even $\mathrm{Id} (L)$ is countable) but does not have property R, the Scott space $\Sigma~\!\!P$ may not be sober. As a complete lattice, the Isbell's non-sober lattice $L$ has property R but does not satisfy any of Conditions (i)-(v) of Theorem \ref{property-R-Scott-sober-cor1}.

Also, for a countable dcpo $P$, Example \ref{examp-Scott-topology-on-product-of-two-Johnstone-dcpos} shows that if $\Sigma~\!\! (P\times P)=\Sigma~\!\!P\times\Sigma~\!\!P$ and $\Sigma~\!\!P$ is coherent but not well-filtered, the Scott space $\Sigma~\!\!P$ may not be sober. Example \ref{exam-Scott-topology-on-product-of-two-Jia-dcpos} shows that if $\Sigma~\!\! (P\times P)=\Sigma~\!\!P\times\Sigma~\!\!P$ and $\Sigma~\!\!P$ is well-filtered but not coherent, the Scott space $\Sigma~\!\!P$ may not be sober.

As we all know, the Johnstone's dcpo $\mathbb{J}$ is a countably infinite dcpo whose Scott
space is non-sober and the Isbell's non-sober complete lattice $L$ is neither distributive nor countable. Thus in 1994, Abramsky and Jung asked
whether there is a distributive complete lattice whose Scott space is non-sober (see \cite[Exercises 7.3.19-6]{Abramsky-Jung-1994} or \cite{Jung-2019}). In \cite{Xu-Xi-Zhao-2021}, using Isbell's lattice, Xu, Xi and Zhao gave a positive answer to this problem.

\begin{theorem}\label{spacial frame Scott is non-sober} \emph{(\cite{Xu-Xi-Zhao-2021})} Let $L$ be the Isbell's non-sober lattice. Then $\mathbf{K}(\Sigma~\!\!L)$ is a spatial frame and the Scott space of $\mathbf{K}(\Sigma~\!\!L)$ is non-sober.
\end{theorem}

In 2019, in a talk \cite{Jung-2019} given in National Institute of Education, Singapore, Achim Jung also posed the following problem.

\vskip 0.1cm

$\mathbf{Problem~4.}$ Is there a countable complete lattice that has a non-sober Scott topology?

\vskip 0.1cm

In \cite{Miao-Xi-Li-Zhao-2022}, Miao, Xi, Li and Zhao constructed a countably infinite complete lattice whose Scott space is non-sober and thus gave a negative answer to Problem 4. The structure of such a countable complete lattice is somewhat complicated. Indeed, if $L$ is a countable complete lattice whose Scott space is non-sober, then by Corollary \ref{Scott-topology-on-product-is-Scott-topology-product-imply-sober} and Corollary \ref{Scott-topology-on-product-of-two-special-countable-dcpos}, there is $x\in L$ such that $\mathcal {I}_x=\{I\in \mathrm{Id}(L) : \vee I=x\}$ is uncountable, whence by Lemma \ref{count-directed-set-sup-count-chain-sup} there are uncountably many strictly ascending chains $C_{\alpha}$ (of course, every $C_{\alpha}$ is countable) with $x$ as their suprema such that the generated ideals $\da C_{\alpha}$ are pairwise different. In \cite{Miao-Xi-Li-Zhao-2022} it was also shown that there is a countable distributive complete lattice whose Scott space is non-sober.

So naturally one asks the following two questions.

\begin{question}\label{countable-cHa-Scott-non-sober-question} Is there a countable frame whose Scott space is non-sober?
\end{question}

\begin{question}\label{countable-spacial-frame-Scott-non-sober-question} Is there a countable spatial frame whose Scott space is non-sober?
\end{question}

\end{document}